\documentclass[12pt,english,reqno]{amsart}
\numberwithin{equation}{section}
\usepackage{bm}
\usepackage{color}
\usepackage{babel}
\usepackage{verbatim}
\usepackage{amsthm}
\usepackage{amstext}
\usepackage{amssymb}
\usepackage{amsbsy}
\usepackage{ifpdf}
\usepackage[nobysame,abbrev,alphabetic]{amsrefs}
\usepackage{enumerate}
\usepackage{amsmath}
\usepackage{amscd}
\usepackage{amsfonts}
\usepackage{graphicx}
\usepackage[all]{xy}
\usepackage{verbatim}
\usepackage{gensymb}
\usepackage{mathrsfs}
\usepackage[colorlinks]{hyperref}
\hypersetup{
linkcolor=blue,          % color of internal links (change box color with linkbordercolor)
citecolor=green,        % color of links to bibliography
}
\newtheorem{theorem}{Theorem}[section]

\newtheorem{lemma}[theorem]{Lemma}
\newtheorem{corollary}[theorem]{Corollary}
\theoremstyle{definition}\newtheorem{definition}[theorem]{Definition}

\theoremstyle{theorem}\newtheorem{proposition}[theorem]{Proposition}

\theoremstyle{definition}

\theoremstyle{definition}
\theoremstyle{definition}\newtheorem{remark}[theorem]{Remark}
\theoremstyle{definition}
\newcommand{\al}{\alpha}

\newcommand{\del}{\delta}
\newcommand{\Del}{\Delta}

\newcommand{\Lam}{\Lambda}
\newcommand{\eps}{\epsilon}

\newcommand{\ka}{\kappa}
\newcommand{\sig}{\sigma}
\newcommand{\Sig}{\Sigma}
\newcommand{\om}{\omega}
\newcommand{\Om}{\Omega}
\newcommand{\vphi}{\varphi}
\newcommand{\cF}{\mathcal{F}}

\newcommand{\cO}{\mathcal{O}}

\newcommand{\bC}{\mathbb{C}}

\newcommand{\bR}{\mathbb{R}}
\newcommand{\bZ}{\mathbb{Z}}
\newcommand{\bQ}{\mathbb{Q}}

\newcommand{\SL}{\operatorname{SL}}

\newcommand{\GL}{\operatorname{GL}}

\newcommand{\Phisym}{\Phi_*}
\newcommand\ceil[1]{\lceil#1\rceil}
\newcommand{\defi}{\overset{\on{def}}{=}}
\newcommand{\comp}{\textrm{{\tiny $\circ$}}}

\newcommand\norm[1]{||#1||}

\newcommand\set[1]{\left\{#1\right\}}
\newcommand\pa[1]{\left(#1\right)}
\newcommand\idist[1]{\langle#1\rangle}

\newcommand\av[1]{|#1|}
\newcommand\on[1]{\operatorname{#1}}
\newcommand\diag[1]{\operatorname{diag}\left(#1\right)}

\newcommand\mb[1]{\mathbf{#1}}

\newcommand\mat[1]{\pa{\begin{matrix}#1\end{matrix}}}
\newcommand\br[1]{\left[#1\right]}
\newcommand\smallmat[1]{\pa{\begin{smallmatrix}#1\end{smallmatrix}}}
\newcommand\crly[1]{\mathscr{#1}}

\newcommand{\spa}{\on{span}}

\newcommand{\dsim}{\Del_*}

\newcommand{\onto}{\xymatrix{\ar@{>>}[r]&}}
\newcommand{\eq}[1]
{
\begin{equation*}
{#1}
\end{equation*}
}
\newcommand{\eqlabel}[2]
{
\begin{equation}
{#2}\label{#1}
\end{equation}
}
\begin{document}
\title{Full escape of mass for the diagonal group}
\author[Uri Shapira]{Uri Shapira}
\begin{abstract}
Building on the work of Cassels
we prove the existence of infinite families of compact orbits of the diagonal group in the space of lattices  
which accumulate only on the divergent
orbit of the standard lattice. As a consequence, we prove the existence of full escape of mass for sequences of such orbits, sharpening 
known results about possible partial escape of mass. The topological complexity of the visits of these orbits to a given compact
set in the space of lattices is analyzed. Along the way we develop new tools to compute regulators of orders in a number fields.
\end{abstract}
%%
%% Authors emails, addresses and acknowledgments
\address{Department of Mathematics\\
Technion \\
Haifa \\
Israel }
\email{ushapira@tx.technion.ac.il}
\thanks{The author acknowledges the support of the Chaya fellowship and ISF grant 357/13.}
\maketitle
\section{Introduction}\label{introduction section}
\subsection{Preliminary description of results}
Let $d\ge2$ be an integer and let $X\defi \SL_d(\bR)/\SL_d(\bZ)$ denote the space of unimodular lattices in $\bR^d$. Let $A<\SL_d(\bR)$ be the group
of diagonal matrices with positive diagonal entries. The aim of this paper is to establish the following result.
\begin{theorem}\label{t.0932}
There exists a sequence of distinct compact $A$-orbits $Ax_k$ such that any limit point of the form $x=\lim a_k x_k$, with $a_k\in A$, 
satisfies $x\in A\bZ^d$.
\end{theorem}
We note that in dimension $d=2$ Peter Sarnak points out in~\cite{Sarnakreciprocal} a construction yielding the above result. In higher 
dimensions the analysis is more intricate.   

Given a compact orbit $Ax\subset X$, let us denote by $\mu_{Ax}$ the unique $A$-invariant probability measure supported on it. 
Theorem~\ref{t.0932} implies
the following result which sharpens~\cite[Theorem 1.10]{ELMV-Duke} in which only partial escape of mass is claimed.
\begin{corollary}\label{t.0901}
Let $Ax_k$ be a sequence of compact orbits satisfying the conclusion of Theorem~\ref{t.0932}. Then the probability measures
$\mu_{Ax_k}$ converge to the zero measure.
\end{corollary}

\begin{proof}
If $\mu$ is an accumulation point of $\mu_{Ax_k}$ then by Theorem~\ref{t.0932}, we conclude that $\mu$ is 
supported on the divergent orbit of the standard lattice.
By Poincar\'e recurrence theorem, the only finite $A$-invariant measure supported on $A\bZ^d$ is the zero measure.
\end{proof}
Our results are much more precise than Theorem~\ref{t.0932} but as a consequence, more lengthy to state and depend on various notation to be presented 
later. We roughly describe their features. For $\eps>0$ we set 
\eqlabel{eq length of vectors}{
X^{\ge\eps}\defi\set{x\in X:\ell(x)\ge \eps}\textrm{ where }\ell(x)\defi\min\set{\norm{v}:0\ne v\in x}.
}

In \S\ref{escape and geometry section} we isolate a family of compact orbits $\cF$, (which in~\S\ref{construction section} is shown to be infinite),
for which there exists a function $\eta:\cF\to(0,\infty)$ with $\eta(Ax)\to 0$ as $Ax\in \cF$ varies, such that for any $\del>0$ and for all but 
finitely many   
$Ax\in\cF$, the following statements are satisfied (see Propositions~\ref{eomprop}, \ref{accprop}, and Theorem~\ref{index theorem} for precise statements):
\begin{enumerate}
\item  $\mu_{Ax}(X^{\ge\del})\le \eta(Ax)$.
\item $\forall y \in Ax\cap X^{\ge\del},\; \on{d}(y, A\bZ^d)\le \eta(Ax)$.
\item The number of connected components of $Ax\cap X^{\ge \del}$ is at least $(d-1)!$.
\end{enumerate}
In fact, the results are sharper in the sense that $\eta(Ax)$ is explicit and more interestingly, $\del$ could be chosen to be an explicit function of 
the orbit such that $\del(Ax)\to 0$ as $Ax\in\cF$ varies.

From the point of view of algebraic number theory,  
Theorem~\ref{index theorem}, which gives new tools for showing that certain collections of units in an order $\cO$ generate
the group of units $\cO^\times$ up to torsion\footnote{This result is only applicable under certain assumptions relating the discriminant of the order
to the geometry of the collection of units.}, might be of interest on its own. This result is (on the face of it) unrelated to 
the discussion on escape of mass but its proof uses the analysis yielding Theorem~\ref{t.0932} in a fundamental way.
For instance,
given a vector of distinct integers $\mb{m}\in\bZ^d$, let $\theta_k$ be a root of the polynomial
$p_k(x)\defi \prod_1^d(x-km_i)-1$. Then, it follows from Theorem~\ref{index theorem} and 
the analysis in \S\ref{construction section}, that 
for all large enough $k$, the collection $\set{\theta_k-km_i}_{i=1}^d$
generates the group of units of the order $\bZ[\theta_k]$ up to roots of unity.

\subsection{Cassels' work on Minkowski's conjecture}
This paper originated from a simple observation made while reading Cassles' paper~\cite{Casselsnforms}
showing that the value $2^{-d}$ is not isolated in the Minkowski spectrum. Given a lattice $x\in X$ let us define 
$$\mu(x)=\sup_{\mb{v}\in\bR^d}\inf\set{\prod_1^d\av{w_i-v_i}:\mb{w}\in x},$$
and let us define the the \textit{Minkowski spectrum}  in dimension $d$ to be $\crly{M}\defi\set{\mu(x) :x\in X}$. A 
famous conjecture attributed to Minkowski asserts that $\crly{M}\subset [0,2^{-d}]$ and that $\mu(x) = 2^{-d}$ 
if and only if $x\in A\bZ^d$. 
  
In~\cite{Casselsnforms} Cassels gives a construction of a sequence $\set{x_k}$ of lattices having compact $A$-orbits
with the property that $\lim_k\mu(x_k)=2^{-d}$. Taking into account the fact that $\mu$ is $A$-invariant and upper semi-continuous
(in the sense that if $x_k\to x$ then $\limsup\mu(x_k)\le \mu(x)$), one arrives at the inevitable conclusion that, assuming Minkowski's conjecture
holds in dimension $d$, the sequence $x_k$ from Cassels' construction must satisfy the conclusion of Theorem~\ref{t.0932}; thus proving it 
in any dimension in which Minkowski's conjecture is known to hold. To this date Minkowski's conjecture has been validated up to dimension $n=9$ 
(see for example~\cite{McMullenMinkowski},\cite{Hans-Gill-dim8},\cite{Hans-Gill-dim9},\cite{SW2} for recent accounts). 

As it turns out, in order to prove Theorem~\ref{t.0932} one does not need Minkowski's conjecture as an input and it could be derived  by 
a careful analysis of Cassels' construction which we slightly generalize in \S\ref{construction section}.

\subsection{Homogeneous dynamics context}
Another context for interpreting the results is the comparison between unipotent dynamics and (higher-rank) diagonalizable dynamics. See
\cite{ELMV-Duke} for a thorough discussion explaining the fundamental differences between these two worlds. In this spirit, we remark that
Theorem~\ref{t.0932} and its corollary are in sharp contrast to the rigidity exhibited by sequences of periodic orbits $Hx_k$
if one assumes $H<\SL_d(\bR)$ is generated by unipotent elements. It is a consequence of a result
of Mozes and Shah~\cite{MozesShah} (relying on results by Dani and Margulis), that if $Hx_k\subset X$ is such a sequence and there is a fixed compact set $K\subset X$ such that
$Hx_k\cap K\ne\varnothing$, then any accumulation point of the sequence $\mu_{Hx_k}$ is a probability measure\footnote{The
information on such accumulation points is much more informative but we state here only the part relevant for us.}. 
We note that there exists
a compact set $K$ such that $Ax\cap K\ne\varnothing$ for any $x\in X$ (see Theorem~\ref{thm del0} and the discussion preceding it for references).

Having this comparison in mind, we find the following open questions (for $d\ge 3$) natural and interesting.
\begin{enumerate}[(Q1)]
\item Does there exists a sequence of compact orbits $Ax_k$ such that $\mu_{Ax_k}$ converges to a non-ergodic measure?
\item Does there exists a sequence of compact orbits $Ax_k$ such that $\mu_{Ax_k}$ exhibits strictly partial escape 
of mass (i.e.\ which converges to a non-zero measure of total mass $<1$)?
\item Does there exists a sequence of compact orbits $Ax_k$ such that $\mu_{Ax_k}$ converges to a 
limit whose ergodic decomposition contains a periodic $A$-invariant measure with positive weight?
\end{enumerate}
Another interesting question is whether or not one can construct a sequence of compact orbits $Ax_k$ with $\mu_{Ax_k}$ exhibiting
full (or even partial) escape of mass such that the lattices $x_k$ are geometric embeddings of full modules in a fixed number field. 
We note though that
in dimension 2 for example,  
due to the results in~\cite{AS}, the arithmetic relations between the full modules in such an example 
must involve infinitely many primes. Interestingly, in positive characteristics this situation breaks and there are examples in dimension
2 of sequences of compact orbits arising from a fixed quadratic field and sharing arithmetic relation involving a single prime which 
produce full escape of mass. For details see~\cite{KPS}.
\subsection{Acknowledgments} This paper was under writing for quite some time during which it evolved to its present shape. 
This evolution was greatly influenced by conversing with others and for that I am grateful. Thanks are due to Elon Lindenstrauss, Manfred Einsiedler, Shahar Mozes, Barak Weiss, and Ofir David.   
\section{Intuition}\label{section intuition}
As often happens, the idea behind the proofs is simple but it is not unlikely that without clarifications 
it may appear quite hidden behind 
the details. We therefore try to explain it briefly and along the way present some of the notation that we shall use. 
The diagonal group is naturally identified with the hyperplane
$$\bR^d_0\defi\set{\mb{t}\in\bR^d:\sum_1^d t_i=0}.$$ 
As the dimension of this hyperplane will appear frequently in our 
formulas, we  henceforth denote
$$n\defi \dim \bR^d_0=d-1.$$
Given a compact orbit $Ax$, it corresponds in a natural way\footnote{For the sake of the current discussion
the reader might envision $\Del$ as the image under the logarithm of the stabilizer or $x$ but in fact,  
the lattice we consider could potentially be slightly larger. 
See \S\ref{escape and geometry section} for the definition.} 
to a lattice $\Del<\bR^d_0$ and 
on choosing a fundamental domain $F$ (say a parallelepiped) for this lattice, we may identify the orbit with $F$ via
the map $\mb{t}\mapsto a(\mb{t})x$, where 
\eqlabel{eq at}{
a(\mb{t})\defi \diag{e^{t_1},\dots, e^{t_d}}.
}
We
think of $\mb{t}\in\bR^d_0$ as the `time parameter', and as it varies, the point on the orbit corresponding to it 
varies as well. 
If it so happens that $x$ contains a very short vector, say of length $\ell$, at time $0$ (and hence 
at any time from $\Del$),
then one can estimate quite easily a radius $r=r(\ell)$ such that for any time in a ball of radius $r$ around a 
corner of $F$ 
the vector is still short. A strategy for deducing that the orbit spends most
of the time near infinity is to establish that 
the volume of these neighbourhoods occupies most of the volume of $F$. The success of this strategy depends therefore
first and foremost on the relationship between the volume of $F$ and the length $\ell$ but also on the  `geometry' of $F$ 
as is illustrated in Figure~\ref{figure basic idea} below.
\begin{figure}[h]
\centering
\fbox{\includegraphics[scale = 0.15]{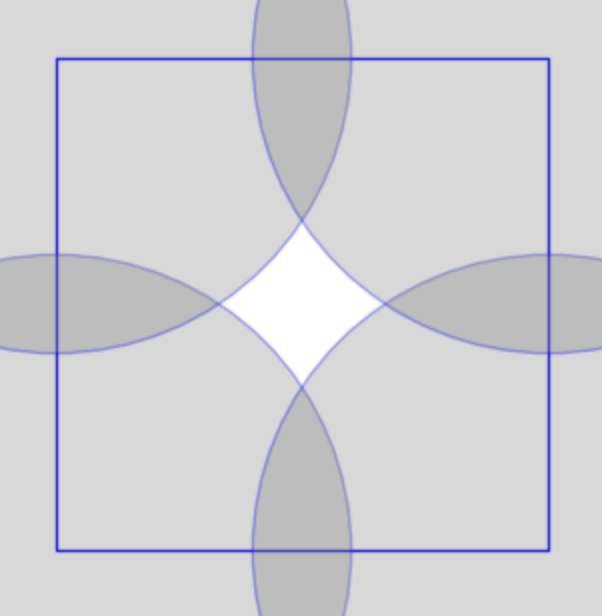}
\includegraphics[scale = 0.20]{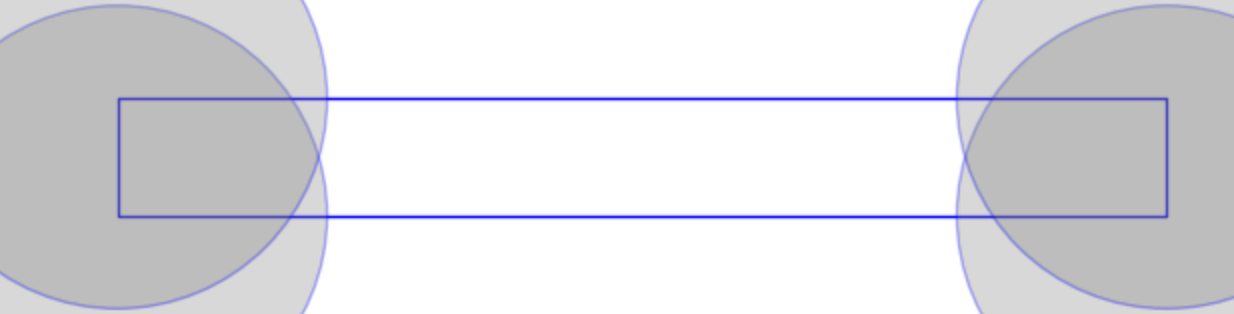}}
\caption{The square on the left is a tamed fundamental domain of area 1 and the grey circles cover most of its area. 
The rectangle on the right is a skewed fundamental domain and the circles fail to cover a significant proportion of it.
The grey areas represents times in which the orbit is near infinity.}\label{figure basic idea}
\end{figure}

This strategy suffices to establish partial escape of mass in certain examples but in order to obtain the sharper 
results of this paper, and in particular, full escape of mass,
one needs to modify it slightly. 

The first modification has to do with the norm. In the above description we mentioned 
an estimate of the radius of a euclidean ball around points of $\Del$ in which we have control on the length of some short vector. It turns out that this estimate becomes sharper when we replace the euclidean norm with a slight modification of it
which takes advantage of the fact that $\bR^d_0$ sits in $\bR^d$ . 
Given $\mb{t}\in\bR^d_0$, we denote 
$$\ceil{\mb{t}}=\max\set{t_i : 1\le i\le d}.$$
Note that there exists a constant $c=c(d)>0$ such that for $\mb{t}\in\bR^d_0$,
 $c^{-1}\norm{\mb{t}}\le \ceil{\mb{t}}\le c\norm{\mb{t}}$. The ball of radius $r$ is then replaced by the set
 $\set{\mb{t}\in\bR^d_0: \ceil{\mb{t}}\le r}$ which as the reader can easily convince himself, is a simplex. Another 
 modification to the above strategy, and which is tightly related to the switch from euclidean balls to simplexes,
  is as follows. Since we have good estimates for the length of vectors when the time parameter is confined to a simplex, it
  will be more convenient not to work with a fundamental domain given by the 
  parallelepiped obtained from a basis of $\Del$, but with a certain simplex which one can cook up
  from this basis. This brings us to a situation where the balls and the rectangles  
  in Figure~\ref{figure basic idea} are replaced respectively by simplexes centered around points of $\Del$ and a simplex of the same type
  which contains a fundamental domain as illustrated in Figure~\ref{figure passing to simplexes} below.
\begin{figure}[h]
\centering
\fbox{
\includegraphics[scale = 0.30]{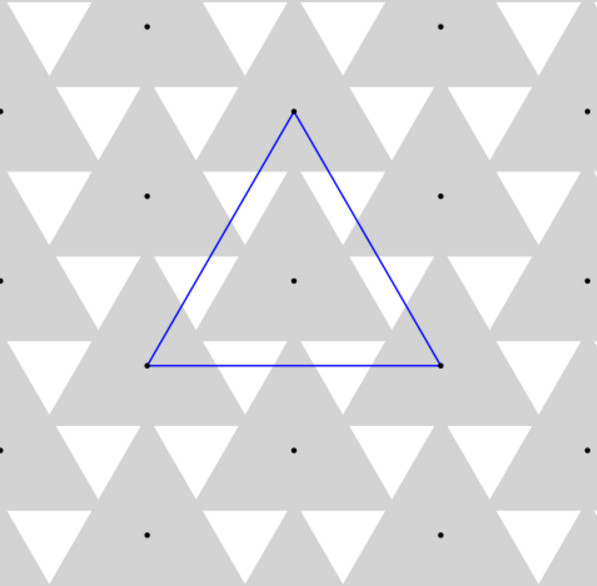}
}
\caption{The black dots are the points of $\Del$, the grey simplexes are the areas where we have control on the short vector, 
and the indicated blue triangle border a simplex which contains a fundamental domain for $\Del$.}
\label{figure passing to simplexes}
\end{figure}  
  
The structure of the paper is as follows. 
In~\S\ref{section simplex sets} we introduce the concept of a \textit{simplex set} $\Phi$ of  a lattice $\Del<\bR^d_0$ 
and various quantities 
attached to it. This section culminates in Proposition~\ref{covering prop} which is fundamental for the analysis. In a nutshell, it 
gives the appropriate mathematics content to the illustration in
Figure~\ref{figure passing to simplexes}.

 In \S\ref{escape and geometry section}
we study lattices with compact orbits of the diagonal group possessing nice simplex sets (which we refer to as $M$-tight simplex sets). 
Having an $M$-tight simplex set allows us translate Proposition~\ref{covering prop} and establish the main results which are Propositions~\ref{eomprop},
\ref{accprop}, and Theorem~\ref{index theorem}. In particular, for these lattices we will control the white area inside the blue simplex
in Figure~\ref{figure passing to simplexes} and show that it is negligible compared to the covolume of $\Del$ 
(see Proposition~\ref{eomprop}\eqref{eq1630}). 
In \S\ref{construction section} we give a construction and carefully analyze it,
showing that the results of \S\ref{escape and geometry section} may be applied. This construction is a slight generalization 
of the one given by Cassels.

\section{Simplex sets and a covering radius calculation}\label{section simplex sets}
\subsection{} 
We begin by fixing some notation and terminology. For concreteness we fix the supremum norm on $\bR^d$ and work only with this norm and the operator norm corresponding to it. 
\begin{comment}
that will be used while discussing lattices in $\bR^d_0\defi\set{\mb{t}\in\bR^d:\sum_1^d t_i=0}$ (which arise naturally when
discussing compact orbits of the diagonal group). Since the dimension of $\bR^d_0$ will appear frequently 
in our discussion we henceforth denote
$$n\defi \dim \bR^d_0=d-1.$$ 
For concreteness we fix the supremum norm on $\bR^d$ and work only with this norm and the operator norm corresponding to it. For $\mb{v}\in\bR^d$
we write
$$\ceil{\mb{v}}=\max\set{v_i : 1\le i\le d}.$$
\end{comment}
By $A\ll B$ we mean that there exists a constant $c$ such that $A\le cB$. If the constant depends on some parameter $M$, we indicate this by writing $\ll_M$. An
exception to this rule is the dependency of the constant in the dimension which we do not record in the notation (thus if not indicated otherwise, implicit constants 
are universal but may depend on $d$). Similar conventions are used with big $O$ and small $o$ notation. We write $A\asymp B$ to indicate that $A\ll B$ and $B\ll A$. For example, as noted in~\S\ref{section intuition}, 
\eqlabel{eq1054}{
\forall \mb{t}\in\bR^d_0,\;\;\norm{\mb{t}}\asymp \ceil{\mb{t}}.
}
\begin{definition}
A \textit{simplex set} $\Phi\subset \bR^d_0$ is a spanning set of cardinality $d$ 
such that $\sum_{\mb{t}\in\Phi}\mb{t}=0$. The \textit{associated simplex} is defined as the convex hull of $\Phi$; $S_\Phi\defi\on{conv}(\Phi)$, and the \textit{associated lattice} is $\Del_\Phi\defi\on{span}_\bZ\Phi$.
\end{definition}
Although the concept of a simplex set
is almost equivalent to a choice of a basis for the associated lattice, it will be more convenient to work with it. 

Similar to the standard lattice and standard basis in euclidean space it will be convenient to introduce a certain ``standard'' simplex set:
Let 
\begin{equation*}
\Phi_*\defi\set{\mb{b}_j}_{j=1}^d,\textrm{ where, }
\mb{b}_j\defi(1,\dots, \underbrace{-n}_{\textrm{\tiny{ $j$'th place}}},\dots,1)^t,
\end{equation*}
and denote $\dsim\defi \Del_{\Phi_*}$. As will be shown in the proof of Lemma~\ref{reduction lemma}, $\Del_*$ is a 
dilated copy of the orthogonal projection of $\bZ^d$ to $\bR^d_0$.
\begin{definition}
Given a simplex set $\Phi$, we define its \textit{distortion map} to be an element $h_\Phi \in \GL(\bR^d_0)$ 
satisfying 
$h_\Phi\Phisym=\av{\Del_\Phi}^{-\frac{1}{n}}\Phi$ 
(and as a consequence $h_\Phi\dsim=\av{\Del_\Phi}^{-\frac{1}{n}}\Del_\Phi$). We define the \textit{distortion} of 
$\Phi$ to be $\norm{h_\Phi}$.
\end{definition}
Note that there is no canonical choice of $h_\Phi$ but the $d!$ possible choices differ by precomposing with a permutation matrix and
thus this ambiguity will cause no harm. 

Most of the assertions regarding a general simplex set $\Phi$ will be proved by establishing them first to $\Phi_*$ and then applying 
the linear map $\av{\Del_\Phi}^{\frac{1}{n}} h_\Phi$ that maps $\Phi_*$ to $\Phi$.

A related quantity to the distortion is the following.
\begin{definition}
Given a simplex set $\Phi$ we set
$$\xi_\Phi\defi \max_{\mb{t}\in\Phi} \ceil{\mb{t}}.$$ 
\end{definition}
We collect a few elementary facts about the concepts introduced above.
\begin{lemma}\label{dist lemma}
\begin{enumerate}
\item\label{dist1}For any simplex set $\Phi$, the distortion of $\Phi$ satisfies $\norm{h_\Phi}\asymp\av{\Del_\Phi}^{-\frac{1}{n}}\xi_\Phi$.
\item\label{dist1.5} If $h_\Phi$ is a distortion map with $\norm{h_\Phi}\le T$ then for any $\mb{t}\in\bR^{d}_0$, 
$\norm{\mb{t}}\ll_T \norm{h_\Phi \mb{t}}$.
\item\label{dist1.6} If $h_\Phi$ is a distortion map with $\norm{h_\Phi}\le T$, then $\xi_\Phi\ll_T \min\set{\norm{\mb{t}}:\mb{t}\in\Phi}$.
\end{enumerate}
\end{lemma}
\begin{proof}
\eqref{dist1}. It is straightforward that
 $\av{\Del_\Phi}^{\frac{1}{n}}\norm{h_\Phi}\asymp \max\set{\norm{\mb{t}}:\mb{t}\in \Phi}$ and so the results follows from the definition of $\xi_\Phi$ 
 and~\eqref{eq1054}.

\eqref{dist1.5} Because the determinant of a distortion map $h_\Phi$ is a constant, (independent of $\Phi$), we have 
that $\norm{h_\Phi^{-1}}\ll \norm{h_\Phi}^n$ and the claim follows.

\eqref{dist1.6}. This follows from~\eqref{dist1}, \eqref{dist1.5}.

\end{proof}

\begin{lemma}\label{xi bound}
Let $\Phi$ be a simplex set. Then, $\av{\Del_\Phi}^{\frac{1}{n}}\ll \xi_\Phi.$ 
%Furthermore, if $K$ is a compact set in the space of lattices in $\bR^{n+1}_0$ and
%$\Del_\Phi\in K$ then 
\end{lemma}
\begin{proof}
For each $\mb{t}\in \Phi$ we have that 
$\norm{\mb{t}}\ll\ceil{\mb{t}}\le\xi_\Phi$ so that $\Del_\Phi$ is spanned by a basis of vectors whose lengths are $\ll\xi_\Phi$. It follows that 
$\av{\Del_\Phi}\ll \xi_\Phi^n$. Extracting the $n$'th root we obtain the desired claim.
\end{proof}

\subsection{}
We proceed by introducing a collection of points $\mb{W}_\Phi$ associated to a simplex set $\Phi$ which will play an important 
role in our considerations.
These points are indicated in red on the boundary of the blue triangle in Figure~\ref{figure W phi} below, which represents $\frac{n}{2}S_\Phi$. The black points represent $\Del_\Phi$ and the grey simplexes represent $\Del_\Phi +(1-\rho)\frac{n}{2}S_\Phi$ for some
$0\le \rho\le 1$.
\begin{figure}[h]
\centering
\fbox{
\includegraphics[scale = 0.22]{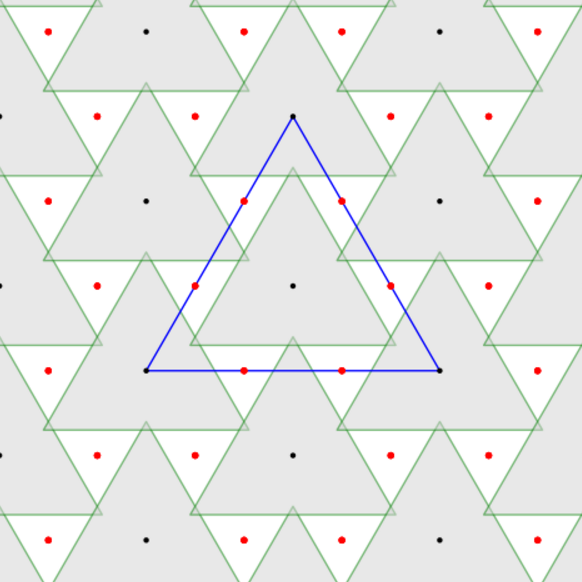}
\includegraphics[scale = 0.22]{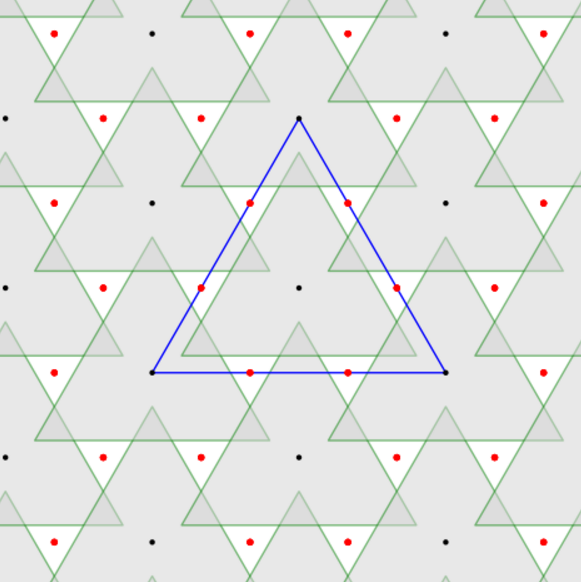}
\includegraphics[scale = 0.22]{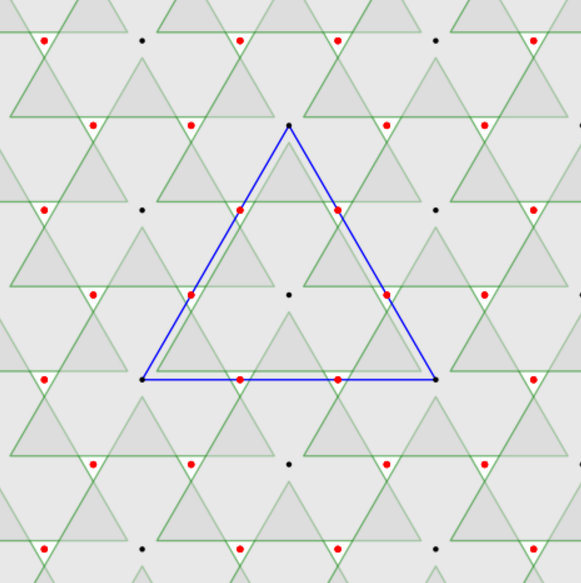}
}
\caption{Envision the grey triangles being dilated simultanuously around their centers. The red points are the critical 
points which are covered last (so called \textit{deep holes}). See Proposition~\ref{covering prop}.}\label{figure W phi}
\end{figure}

Let $\crly{S}$ denote the group of permutation matrices in $\GL_d(\bR)$ and let $\theta\in\crly{S}$ denote the standard cycle permutation; that is, $\theta\mb{e}_i=\mb{e}_{i+1 \on{mod} d}$, where $\set{\mb{e}_i}_{i=1}^d$ denotes the standard basis for $\bR^d$.

\begin{definition}
For any $\tau\in \crly{S}$ let
$$\mb{w}_\tau^{\Phisym}\defi\tau (d^{-1}\sum_{\ell=1}^d (\ell-1) \mb{b}_\ell)$$
and in turn, for a general simplex set $\Phi$ with distortion map $h_\Phi$, let
\eqlabel{def:w}{
\mb{w}_\tau^\Phi\defi \av{\Del_\Phi}^{\frac{1}{n}} h_\Phi \mb{w}_\tau^{\Phisym}.
}
\end{definition}
Note that the indexation in~\eqref{def:w} depends on the choice of the distortion map but this should cause no confusion. We further let $\mb{W}_\Phi\defi \set{\mb{w}^\Phi_\tau}_{\tau\in\crly{S}}$ and choosing a set $\crly{S}'\subset \crly{S}$ of representatives of the quotient $\crly{S}/\idist{\theta}$ we 
set\footnote{The collections $\crly{S}', \mb{W}_\Phi'$ are not canonically defined.} 
$\mb{W}_\Phi'\defi\set{\mb{w}_\tau^\Phi}_{\tau\in\crly{S}'}$. 

When $\Phi$ is understood we usually write $\mb{w}_\tau$. We remark that given a simplex set $\Phi$, the collection $\mb{W}_\Phi$ lies on the boundary of the dilation $\frac{n}{2}S_\Phi$ because $\frac{2}{n}\mb{w}_\tau$ is a convex
combination of the elements in $\Phi$ with one of the coefficients being $0$.
%
\begin{comment}
\begin{figure}[h]
\centering
\fbox{\includegraphics[scale = 0.15]{figure1}}
\fbox{\includegraphics[scale = 0.15]{figure2}}
\fbox{\includegraphics[scale = 0.15]{figure3}}
\caption{This figure illustrates in grey the set $(1-\rho)\frac{n}{2}S_\Phi$ for $\rho = 0.3, 0.1, 0$ respectively. The points of $\Del_\Phi$ are marked in black and the other triangles are translates of the grey one by the lattice points. The red points are the critical points, 
$\mb{W}_\Phi+\Del_\Phi$. What makes them critical is that only when $\rho=0$ the are covered by the translates of the grey triangle (see Proposition~\ref{covering prop}). 
These critical points decompose into $n!$ cosets of $\Del_\Phi$ (see Lemma~\ref{reduction lemma}). See also Figure~\ref{figure2}.}
\label{figure1}
\end{figure}
\end{comment}

%

The following lemma analyzes the reduction of $\mb{W}_\Phi$ modulo $\Del_\Phi$.
\begin{lemma}\label{reduction lemma}
Let $\Phi$ be a simplex set. 
\begin{enumerate}
\item\label{claim4} For $\tau,\sig\in\crly{S}$, $\mb{w}_\tau+\Del_\Phi=\mb{w}_\sig+\Del_\Phi$ if and only if $\tau^{-1}\sig\in\idist{\theta}$.
\item\label{claim5} If $\tau^{-1}\sig\notin\idist{\theta}$ then 
there exists $\tau'\in\crly{S}$ such that for some  $1\le k\le n$,
$\mb{w}_{\tau'}+k(\mb{w}_\tau-\mb{w}_\sig)\notin
\mb{W}_\Phi+\Del_\Phi.$ Moreover, if $d$ is prime then any $\tau'$ satisfies this.
\end{enumerate}
\end{lemma}
Statement~\eqref{claim4} is illustrated in Figure~\ref{figure W phi} where one can see that the six red points on the boundary of the blue 
simplex decompose into two classes modulo the black points.
Statement~\eqref{claim5} will be used in the proof of Theorem~\ref{index theorem} 
to show that a certain lattice that a priori contains
$\Del_\Phi$ is actually equal to it. This will allow us to prove in~\S\ref{construction section} that certain collections of units in certain
orders generate the corresponding group of units modulo torsion.
\begin{proof}
It suffices to prove the lemma for the standard simplex set $\Phisym$. 
Let us denote by $\on{p}:\bR^d\to \bR^d_0$ the orthogonal projection.
Note that if $\set{\mb{e}_\ell}$ denotes the standard basis for $\bR^d$ then for any $\ell$ and for any $\sig\in\crly{S}$
\begin{align}
\nonumber \mb{b}_\ell &= \on{p}(-d\mb{e}_\ell),\\
\label{eq1750}\dsim &=\on{p}(d\bZ^d),\\
\nonumber \mb{w}_\sig&=\on{p}((\sig(1,\dots, d))^t).
\end{align}
 
We first show that for any $\tau$, $\mb{w}_\tau - \mb{w}_{\tau\comp \theta^j}\in\dsim$. By the above, this is equivalent to
showing  
that the difference $\tau(1,\dots,d)-\tau\comp\theta^j(1,\dots,d)\in d\bZ^d+\ker\on{p}$. Clearly, it is enough to deal 
with the case $j=1$ and moreover, post-composing with $\tau^{-1}$ we may assume $\tau=\on{id}$ (here we are using 
the fact that both $\bZ^d$ and $\ker\on{p}$ are $\crly{S}$-invariant). 
Indeed,  
$$ \smallmat{1\\2\\ \vdots \\d} -\theta \smallmat{1\\2\\ \vdots \\ d}
=\smallmat{1\\1\\ \vdots \\1}-\smallmat{d\\0\\ \vdots \\ 0}\in\ker\on{p}+d\bZ^d.$$

We prove the second part of~\eqref{claim4}. Let $\tau,\sig\in\crly{S}$ be such that $\tau^{-1}\sig\notin\idist{\theta}.$ 
We show that $\mb{w}_\tau-\mb{w}_\sig\notin\dsim$, which is equivalent to $\tau(1,\dots,d)-\sig(1,\dots,d)\notin d\bZ^d+\ker\on{p}$. Post-composing with $\tau^{-1}$ and rewriting $\tau^{-1}\sig$ as $\sig$ we are reduced to the case $\tau=\on{id}$. Moreover, using the already established first part of~\eqref{claim4}, we may replace $\sig$ with $\sig\comp\theta^j$ and we choose $j$ so that $\sig\comp\theta^j$ fixes the $d$'th coordinate. To this end, assume by way of contradiction
the existence of a non-trivial permutation $\sig$ of $1,\dots,n$ such that 
the difference $(1,\dots,n,d)-(\sig(1),\dots,\sig(n),d)$  belongs to 
$d\bZ^d+\ker\on{p}$. Alternatively, all the differences between the coordinates of $(1-\sig(1),\dots,n-\sig(n),0)$ are 
equal modulo $d$. As the last coordinate is $0$ this is equivalent to saying that $(1,\dots,n)=(\sig(1),\dots,\sig(n))$ modulo
$d$. But since this equality involves only numbers between 1 to $n$ it is equivalent to $\sig$ being trivial which contradicts our 
assumption and concludes the proof of~\eqref{claim4}.

The proof of~\eqref{claim5} follows similar lines. First observe that because $\ker\on{p}\cap \bZ^d=\set{k\mb{1}:k\in\bZ}$, where 
$\mb{1}\defi (1,\dots,1)^t$, it follows from~\eqref{eq1750} that $\set{\mb{v}\in\bZ^d:\on{p}(\mb{v})\in\mb{W}_\Phi}$ can be characterized as the set of integer vectors having all residues modulo $d$ as coordinates. 

Similarly to the consideration above we 
reduce to the case where $\tau=\on{id}$ and $\sig$ is non-trivial and  fixes the $d$'th axis. Thus we need to find $\tau'$ and $k$ such that
$\tau'(1,\dots,d)+k(1-\sig(1),\dots, n-\sig(n),0)$ does not contain all the residues modulo $d$ as coordinates. Since $\sig$ is non-trivial we know that for some $1\le j\le n$, $j-\sig(j)\ne 0$. If we choose
$\tau'$ so that $\tau'$ fixes the $d$ axis and also, $\tau'(j) =-(j-\sig(j))$ modulo $d$ then we see that for $k=1$ the above vector will contain the residue $0$ twice and therefore will not contain all residues. If $d$ is assumed to be prime, then
for any $\tau'$ we could solve the equation $\tau'(d) = \tau'(j)+k(j-\sig(j))$ modulo $d$ with some $1\le k\le d$, and then
the above vector will contain the residue $\tau'(d)$ twice yielding the same contradiction. 

We note that when $d$ is not prime it 
 is possible to find $\tau',\tau,\sig$ such that $\tau^{-1}\sig\notin \crly{S}$ but
$\mb{w}_{\tau'}+k(\mb{w}_\tau-\mb{w}_\sig)\in\mb{W}_\Phi+\Del$ for all $k$. For example, 
when  $d=4$ one can take
$\mb{w}_{\tau'}=\mb{w}_\tau = (3,2,1,4), \mb{w}_\sig = (1,2,3,4)$. 
\end{proof}

We now state a covering result which plays a crucial role in the analysis below and shows the significance of the 
dilation $\frac{n}{2}S_\Phi$ and the collection
$\mb{W}_\Phi$ (see Figure~\ref{figure W phi}). This is roughly a restatement of \cite[Lemma 1]{Casselsnforms} but we give the full proof for 
completeness.
\begin{comment}
\begin{figure}[h]
\centering
\fbox{\includegraphics[scale = 0.2]{figure11}}
\fbox{\includegraphics[scale = 0.2]{figure21}}
\caption{In this figure the grey area represents $\Del_\Phi + (1-\rho)\frac{n}{2}S_\Phi$ (for $\rho=0.3,0.1,$ respectively), and the white area represent the complement which equals, as is illustrated, to the union of simplexes centred at the points in 
$\Del_\Phi +\mb{W}_\Phi$ and which have side-length proportional to $\rho$. See also Figure~\ref{figure1}.}
\label{figure2}
\end{figure}
\end{comment}

\begin{proposition}\label{covering prop}
Let $\Phi$ be a simplex.  
\begin{enumerate}
\item\label{claim1} $\bR^d_0= \frac{n}{2} S_\Phi +\Del_\Phi$. 
\item\label{claim2} $\bR^d_0\smallsetminus(\frac{n}{2} S_\Phi^\circ+\Del_\Phi)\subset \mb{W}_\Phi+\Del_\Phi=\mb{W}_\Phi'+\Del_\Phi$, where 
$S^\circ$ denotes the interior of $S$.
\item\label{claim3} There exists a universal constant $c>0$ 
such that for any $\rho\in(0,1)$, 
$\bR^d_0\smallsetminus((1-\rho)\frac{n}{2} S_\Phi+\Del_\Phi)\subset \mb{W}_\Phi' + B_{c\rho\xi_\Phi}+\Del_\Phi.$
\end{enumerate}
\end{proposition}
As explained in Remark~\ref{equality remark}, the inclusion in~\eqref{claim2} is in fact equality. This shows that the covering radius of $\Del_\Phi$ by $S_\Phi$
is exactly $\frac{n}{2}$. 
\begin{proof}
We prove the proposition for the  standard simplex set  $\Phisym$.
In this case, it follows from the definition (or from~\eqref{eq1750}), that  $\mb{w}_\tau=\tau(\mb{w})$, where $\mb{w}=(n/2,n/2-1,...,-n/2)$. 
The statement for general simplex sets follows by applying the linear map $\av{\Del_\Phi}^{\frac{1}{n}}h_\Phi$ which takes 
$\Phi_*$ to $\Phi$. Indeed, it is clear that 
the statements \eqref{claim1}, \eqref{claim2}, follow immediately and \eqref{claim3} follows from the fact that the 
above linear transformation has norm $\ll\xi_\Phi$ by Lemma~\ref{dist lemma}\eqref{dist1}.

We prove~\eqref{claim1} and \eqref{claim2}.  
One of the benefits of working in the hyperplane $\bR^d_0$ is the following observation: 
\eqlabel{eq1444}{
\frac{n}{2}S_{\Phisym}=\set{\mb{v}\in\bR^d_0: \ceil{ \mb{v}}\le \frac{n}{2}}.
} 
Given $\mb{u}\in\bR^d_0$, if $\ceil{\mb{u}}\ge\frac{n}{2}$ we describe a procedure by which 
\begin{enumerate}[(i)]
\item\label{pos2} either $\mb{u}=\mb{w}_\tau$ for some $\tau\in\crly{S}$,
\item\label{pos1} or we find $\mb{v}\in\dsim$ such that 
 $\ceil{\mb{u}+\mb{v}}<\ceil{\mb{u}}.$ 
\end{enumerate} 
As long as \eqref{pos1} occurs, we replace $\mb{u}$ by
$\mb{u+v}$ and iterate. 
As the set $\dsim+\mb{u}$ is discrete and the sets $\set{\mb{v}\in\bR^d_0: \ceil{\mb{v}}\le t}$ are compact, we cannot iterate this procedure indefinitely, which means that after finitely many iterations we will arrive at a 
vector $\mb{u}'\in\mb{u}+\dsim$ with $\mb{u}'\in \frac{n}{2}S^\circ$, or to one of the points $\mb{w}_\tau$. 
This will conclude the 
proof of~\eqref{claim1} and  \eqref{claim2} (the equality $\mb{W}_{\Phi_*}+\Del_*=\mb{W}'_{\Phi_*}+\Del_*$ follows from Lemma~\ref{reduction lemma}\eqref{claim4}).

The procedure is as follows:
Let $\mb{u}\in\bR^d_0$ be such that $\max \mb{u} \ge \frac{n}{2}$ and assume for concreteness that $u_1\ge u_2\dots \ge u_d$. 
Successively for $k=1,\dots, n$ look at the vector
$$ \mb{u}^{(k)}\defi \mb{u}+\sum_{j=1}^{k} \mb{b}_j =  
\smallmat{u_1\\ \vdots\\ u_k\\ u_{k+1}\\ \vdots \\ u_d}+
\smallmat{-d+k\\ \vdots \\ -d+k\\ k\\ \vdots\\ k}= 
\smallmat{u_1 - (d-k)\\ \vdots\\ u_k - (d-k)\\ u_{k+1}+k\\ \vdots \\ u_{d}+k}.$$
If for 
some $1\le k\le n$ we have that $\max \mb{u}^{(k)}<\max \mb{u}$ then we are done. Otherwise,
by our assumption on the order of the coordinates of $\mb{u}$ we conclude that $\ceil{ \mb{u}^{(k)}}$ must be obtained at its 
$k+1$'th coordinate  $u_{k+1}+k$ and moreover that $u_1\le u_{k+1}+k$. Summing these $n$ inequalities and taking into account that 
$\mb{u}\in\bR^{n+1}_0$ we arrive at $(n+1)u_1\le \frac{n(n+1)}{2}$, and so $u_1\le \frac{n}{2}$. Together with our assumption that
$\ceil{ \mb{u}}\ge \frac{n}{2}$ we conclude that in fact $u_1=\frac{n}{2}$. Thus all the inequalities in the above consideration must be
equalities. In particular, $u_1= u_{k+1}+k$ for all relevant $k$ which implies that $\mb{u}=\mb{w}$ as desired.

Claim \eqref{claim3} follows from \eqref{claim2} and the following elementary statement about intersections of half-spaces: 
Given $d$ linear functionals $\psi_i$ on $\bR^d_0$ and a point $\mb{w}\in\bR^d_0$, consider the half spaces
$$H_i^\rho\defi\set{\mb{v}\in\bR^d_0:\psi_i(\mb{v})\le \psi_i(\mb{w})+\rho}$$ for $\rho\ge 0$. If the $\psi_i$ are in
general position (in fact, if and only if), then the intersection $\cap_1^dH_i^\rho$ is a simplex which is contained in 
the ball $B_{c\rho}(\mb{w})$ for some constant $c$ that depends only on the functionals $\psi_i$ (this constant deteriorates when 
$n$ of these functionals are almost linearly dependent). In the case at hand the linear functionals are those giving rise for the faces 
of $S_{\Phi_*}$ and $\mb{w}$ is a point of $\mb{W}_{\Phi_*}+\Del_*$. In Figure~\ref{figure W phi} these are the white simplexes 
centered at the red points.
\end{proof}
\begin{remark}\label{equality remark}
A nice indirect way to see why there must be equality in Proposition~\ref{covering prop}\eqref{claim2} is as follows. When considering $\Phi_*$, 
since all the objects involved are invariant under $\crly{S}$, it is clear that the established inclusion implies that the left hand side of~\eqref{claim2}
is either empty or there is equality. Assuming it is empty, we deduce the existence of a positive $\rho>0$ such that the left hand side of~\eqref{claim3} is empty
as well. As will be shown in \S\ref{escape and geometry section}, \ref{construction section}, there are compact $A$-orbits whose visits to $X^{\ge\del}$ occur 
only at times $\mb{t}$ belonging to the left hand side of~\eqref{claim3}. We arrive at the desired contradiction as this set is never empty by Theorem~\ref{thm del0}.
\end{remark}

\section{Escape and geometry of compact $A$-orbits}\label{escape and geometry section}
\subsection{}
Denote
\eqlabel{Omega}{
 \Om\defi\set{x\in X_d: Ax \textrm{ is compact}}.\\
}
Let $U$ denote the group of diagonal matrices with $\pm 1$ diagonal entries so that $UA$ is the group of all diagonal 
matrices of determinant $\pm 1$. Let $\vphi:UA\to \bR^d_0$ be the homomorphism given by $\vphi(\diag{a_1,\dots,a_d})=(\log\av{a_1},\dots,\log\av{a_d})^t$. Finally,  for $x\in\Om$ we let\footnote{Note that $\av{\Del_x}$
is (up to a constant factor) the regulator of some order in a totally real number field.
} 
\eqlabel{deltax}{
\Del_x\defi \vphi(\on{stab}_{UA}(x)).}
Note that $\Del_x$ might strictly contain $\log(\on{stab}_A(x))$. 
In practice, $\mb{u}\in\Del_x$ if and only if there is a $\pm 1$ diagonal matrix $u$ such that $ua(\mb{u})x=x$ (see~\eqref{eq at} for the notation), and since $U$ acts 
by isometries on $\bR^d$ and we will only be interested in estimating length of vectors, the presence of $u$ will have no effect. 

The results to be proved are stated in an effective way with respect to the covolume $\av{\Del_x}$ and thus they become meaningful
only when $\av{\Del_x}\to\infty$. That this is indeed the case when $Ax$ varies is a consequence of the basic fact from algebraic 
number theory that says that the regulator of a number field goes to infinity when the number field varies. We take this fact for granted.

A simplex set 
$\Phi\subset \Del_x$ will be referred to as a \textit{simplex set 
for} $x$. Note that we do not assume that $\Phi$ generates $\Del_x$ and in fact, one of our goals (see Theorem~\ref{index theorem}), is to develop
tools that will enable to prove that under certain assumptions $\Phi$ generates $\Del_x$; a significant fact from the number 
theoretical point of view as it allows to compute regulators. 

\subsection{}
In this section, we extract effective information about a compact orbit $Ax$ assuming that (i) $x$ has a very short vector, 
(ii) there exists a simplex set for $x$ satisfying various 
assumptions related to the length of the short vector. 
The results originates from comparing the results of \S\ref{section simplex sets} and the following basic fact
(recall the notation of~\eqref{eq length of vectors}).
\begin{theorem}\label{thm del0}
There exists a universal constant $\del_0>0$ such that for any $x\in X$, $Ax\cap X^{\ge\del_0}\ne\varnothing$.
\end{theorem}
The proof of Theorem~\ref{thm del0} goes back to the solution of Mordell's inverse problem by Siegel and by
Hlawka~\cite{Hlawka-inverse-problem} (see also~\cite{DavenportNote}). The largest possible value of $\del_0$ 
(which changes with the choice of norm on $\bR^d$) is an interesting quantity in the geometry of numbers. A proof
of Margulis to Theorem~\ref{thm del0} may be found in the appendix of~\cite{TW} and recently it was proved 
in~\cite{McMullenMinkowski}, \cite{SW2}, that for the euclidean norm one can take $\del_0\ge 1$.

We begin with interpreting Proposition~\ref{covering prop}\eqref{claim1} in this context.
\begin{lemma}\label{lemma orbit}
Let $x\in\Om$ and $\Phi$ be a simplex set for $x$. Then $Ax\subset U\set{a(\mb{t}x):\mb{t}\in \frac{n}{2}S_\Phi}.$
\end{lemma}
\begin{proof}
Given any $\mb{t}\in\bR^d_0$, by Proposition~\ref{covering prop}\eqref{claim1} there exists $\mb{u}\in\Del_x$
with $\mb{t}'\defi \mb{t}-\mb{u}\in\frac{n}{2}S_\Phi$ and so $a(\mb{t})x=a(\mb{t}')a(\mb{u})x=ua(\mb{t}')x$ for some $u\in U$.
\end{proof}
Recall the notation of~\eqref{eq length of vectors}. The following observation gives a lower bound for $\ell(x)$ in terms of the geometry of a simplex set.
\begin{proposition}\label{length bound prop}
Let $x\in\Om$ and let $\Phi\subset \Del_x$ be a simplex set. Then, $e^{-\frac{n}{2}\xi_\Phi}\ll \ell(x)$.
\end{proposition}
\begin{proof}
Let $v\in x$ attain the length $\ell(x)$ 
and let $\mb{t}\in\frac{n}{2}S_\Phi$. We give an upper bound
for $\norm{a(\mb{t})v}$. Write
$\mb{t}=\frac{n}{2}\sum_1^{n+1} \al_j\mb{t}_j$ with $\al_j$ coefficients of a convex combination.
Using the definition of $\xi_\Phi$ and the convexity of the exponential map we get
\begin{align}\label{eq238}
\nonumber\norm{a(\mb{t})v}&\le\ell(x) \max_i \pa{e^{\sum_j \al_jt_{ij}}}^{\frac{n}{2}}\\
&\le \ell(x)\max_i \pa{\sum_j\al_j e^{t_{ij}}}^{\frac{n}{2}}\le  \ell(x)e^{\frac{n}{2}\xi_\Phi}.
\end{align}
Theorem~\ref{thm del0} and Lemma~\ref{lemma orbit} imply that there exist $u\in U$ and $\mb{t}\in \frac{n}{2}S_\Phi$ 
such that $\del_0\le\norm{ua(\mb{t})v}=\norm{a(\mb{t})v}$. This together with \eqref{eq238} gives $ e^{-\frac{n}{2}\xi_\Phi}\ll\ell(x)$ as desired.
\end{proof}
Motivated by Proposition~\ref{length bound prop} we make the following definition.
\begin{definition}\label{defomegaM}
For $M\ge1$ say that a simplex set $\Phi\subset \Del_x$ is $M$-\textit{tight} for $x$, if $\ell(x)\le M e^{-\frac{n}{2}\xi_\Phi}$. 
We let
\begin{align}\label{OmegaM}
\Om_M\defi\set{x\in \Om:\textrm{ there exists an $M$-tight simplex set for $x$}}.
\end{align}

\end{definition}
The following reinterprets Proposition~\ref{covering prop} in this context and establishes an effective escape of 
mass statement for lattices in $\Om_M$ in terms of the regulator $\av{\Del_x}$. It is the tool with which one can 
prove Theorem~\ref{t.0901} directly and not as a corollary to Theorem~\ref{t.0932}.
\begin{proposition}\label{eomprop}
Fix $M>1$ and  $\ka\in(0,1)$ and for $x\in \Om_M$ with an $M$-tight simplex set $\Phi$, set
\begin{align}\label{functions}
\nonumber 
&\del_\ka(x)=\del \defi Me^{-\frac{n}{2}\av{\Del_x}^{\frac{\ka}{n}}};\\
&\rho_\ka(x)=\rho\defi\frac{\av{\Del_x}^{\frac{\ka}{n}}}{\xi_\Phi} \ll\av{\Del_x}^{-\frac{1-\ka}{n}};\\
\nonumber
&r_\ka(x)=r\defi c\av{\Del_x}^{\frac{\ka}{n}} \textrm{ ($c$ being as in Proposition~\ref{covering prop})}.
\end{align}
Then,
\begin{enumerate}
\item\label{eq1151}
$\set{a(\mb{t})x : \mb{t}\in\frac{(1-\rho)n}{2}S_\Phi+\Del_\Phi}\subset X^{<\del}$.
\item\label{eq1651}
$\set{\mb{t}\in\bR^{n+1}_0:a(\mb{t}x)\in X^{\ge \del}} \subset \mb{W}'_\Phi+B_{r}+\Del_\Phi$.
%\item\label  $Ax\cap X^{\ge \del}\subset\set{a(\mb{t})x : \mb{t}\in \bigcup_{\tau\in\crly{S}} B_{r_x}(\mb{w}_\tau)}$.
\item\label{eq1630} $\mu_{Ax}(X^{\ge\del})\ll \av{\Del_{x}}^{-1+\ka}$.
\end{enumerate}
\end{proposition}
\begin{remark}
Hereafter we use the symbols $\del,\rho,r$ as in~\eqref{functions} recording implicitly the dependence on $x,\ka$.
\end{remark}
\begin{proof}
Let $x\in\Om_M$ and $\Phi$ an $M$-tight simplex set for $x$.   
Note that by Lemma~\ref{xi bound}, 
$\rho\ll\av{\Del_x}^{-\frac{1-\ka}{n}}$.

We prove~\eqref{eq1151}.  By the discussion following~\eqref{deltax} it suffices to prove 
$\set{a(\mb{t})x : \mb{t}\in\frac{(1-\rho)n}{2}S_\Phi}\subset X^{<\del}$.
Let
$\mb{t}=(1-\rho)\frac{n}{2}\sum_1^{n+1} \al_j\mb{t}_j\in (1-\rho)\frac{n}{2}S_\Phi$ (with $\al_j$ coefficients of a convex combination).
Let $v\in x$ be a vector attaining $\ell(x)$. Similarly to \eqref{eq238}
we get
\begin{align}\label{eq05}
\nonumber \norm{a(\mb{t})v}&\le \ell(x)\max_i \pa{e^{\sum_j \al_jt_{ij}}}^{(1-\rho)\frac{n}{2}}
\le Me^{-\frac{n}{2}\xi_\Phi}\max_i \pa{\sum_j\al_j e^{t_{ij}}}^{(1-\rho)\frac{n}{2}} \\
&\le M e^{-\frac{n}{2}\xi_\Phi}\pa{e^{\xi_\Phi}}^{(1-\rho)\frac{n}{2}}= Me^{-\frac{n}{2}\av{\Del_x}^{\frac{\ka}{n}}}=\del_\ka(x),
\end{align}
which shows that $a(\mb{t}x)\in X^{<\del}$.

We prove~\eqref{eq1651}. By~\eqref{eq1151} we have that 
$$\set{\mb{t}\in\bR^{n+1}_0:a(t)x\in X^{\ge\del}}\subset \bR^{n+1}_0\smallsetminus\pa{(1-\rho)\frac{n}{2}S_\Phi+\Del_\Phi},$$ 
and by
Proposition~\ref{covering prop}\eqref{claim3},
$$\bR^{n+1}_0\smallsetminus\pa{ (1-\rho)\frac{n}{2}S_\Phi+\Del_\Phi}\subset \pa{\mb{W}'_\Phi+B_{c\rho\xi_\Phi}+\Del_\Phi},$$
which implies~\eqref{eq1651} as $r=c\rho\xi_\Phi$.

%\eqlabel{eqballs}{
%Ax\cap X^{\ge \del}\subset\set{a(\mb{t})x : \mb{t}\in \bigcup_{\tau\in\crly{S}} B_{c'\rho \xi_\Phi}(\mb{w}_\tau)},
%}
%which is the statement in \eqref{eq1651} with $r_x\defi c'\av{\Del_x}^{\frac{\ka}{n}}$ (for our choice $\rho=\frac{\av{\Del_x}
%^{\frac{\ka}{n}}}{\xi_\Phi}$). 
Statement~\eqref{eq1630} now follows from~\eqref{eq1651}:
\footnote{Note that the index of $\log(\on{stab}_A(x))$ in $\Del_x$ is
bounded by $2^d.$.}
$$\mu_{Ax}( X^{\ge \del})\ll r^n/\av{\Del_x}\ll \av{\Del_x}^{-1+\ka}.$$
\end{proof}
The following result is the one that we use in order to prove Theorem~\ref{t.0932}. It assumes that we have accurate knowledge regarding the position of the orbit $Ax$ at the special times $\mb{W}_\Phi$ and concludes that the orbit collapses in an effective manner on the orbit $A\bZ^d$. In order to state it we need the following definition:
\begin{definition}\label{omega epsilon c}
For $\eps>0, M,C>1$ we denote by 
$\Om_M(\eps,C)$ the set of $x\in \Om_M$ such that there exists an $M$-tight simplex set $\Phi$ for $x$ with respect to which, for any $\tau\in\crly{S}$
\eqlabel{eq1156}{
a(\mb{w}_\tau)x = g_\tau a(\mb{s}_\tau)\bZ^{n+1} .
} 
where $g_\tau$ satisfies $\norm{g_\tau - I}\le Ce^{-\av{\Del_x}^\eps}$ and $s_\tau\in\bR^d_0$.
\end{definition}
\begin{proposition}\label{accprop}
Fix $M,C>1$ and $\eps>0$. 
Then, for any $\ka<n\eps$ there exists $C'>0$ such that for any $x\in\Om_M(\eps,C) $, 
$$\forall y\in Ax\cap X^{\ge\del},\;  \on{d}(y,A\bZ^d)\le C'e^{-\frac{1}{2}\av{\Del_x}^\eps}.$$
\end{proposition}
\begin{proof}
Let $x\in\Om_M(\eps,C)$ and let $\Phi$ be an $M$-tight simplex set as in Definition~\ref{omega epsilon c}. By Proposition \ref{eomprop}\eqref{eq1651} 
for any point $y\in Ax\cap X_{n+1}^{\ge\del}$ there exists $\tau\in\crly{S}$ such that $y=a(\mb{t}+\mb{w}_\tau)x$ 
for some $\mb{t}\in 
\mb{w}_\tau+B_r$. By~\eqref{eq1156} it follows that 
\eq{
y=a(\mb{t})a(\mb{w}_\tau)x=a(\mb{t})g_\tau a(-\mb{t})a(\mb{t}+\mb{s}_\tau)\bZ^{n+1}
}
and so the distance between $y$ and $A\bZ^{n+1}$ is 
$\le \on{d}(I, a(\mb{t})g_\tau a(-\mb{t}))\le  Ce^{2r_x-\av{\Del_x}^\eps}=Ce^{2 c\av{\Del_x}^{\frac{\ka}{n}}-\av{\Del_x}^\eps}=
Ce^{(-1+2c\av{\Del_x}^{\frac{\ka}{n}-\eps})\av{\Del_x}^\eps}$,
and if $\ka<n\eps$ then $2c\av{\Del_x}^{\frac{\ka}{n}-\eps}<\frac{1}{2}$ for large values of $\av{\Del_x}$ and so for
a suitable $C'$ the latter expression is $\le C'e^{-\frac{1}{2}\av{\Del_x}^\eps}$ as desired.
\end{proof}
\subsection{}
In applications one exhibits a simplex set $\Phi$ for $x\in\Om$ and it is desirable to be able to determine that it generates $\Del_x$. 
In this section we prove Theorem~\ref{index theorem} establishing this in some cases. This theorem also discusses the 
topological complexity of $Ax\cap X^{\ge\del_1}$.
 
Our results rely on an assumption regarding the lattice $\Del_\Phi$ when considered as a point in the space of lattices in
$\bR^d_0$ (identified up to homothety). More precisely, we will have to assume that $\Del_\Phi$ belongs to a given compact set 
in that space (it is not clear at the moment if this assumption is necessary). Below we refer to the aforementioned space of lattices
as the space of \textit{shapes} in order to distinguish it from $X$.
 
It might be worth mentioning here that understanding the set of shapes $\set{\Del_x:x\in\Om}$  is a highly non-trivial question with virtually no results (although in the next section we will see that the shape $\Del_*$ is an accumulation
point of this set).
\begin{definition}\label{omegamk}
Given $M>1$ and $K$ a compact set of shapes we denote by,
$\Om_{M,K}$ the set of $x\in \Om_M$ possessing an $M$-tight simplex set $\Phi$ for $x$ with shape $\Del_\Phi\in K$.
\end{definition}
Because $\Del_\Phi$ is obtained from $\dsim$ by applying a distortion map $h_\Phi$ followed by a dilation, it is obvious from the definition that the set $\set{\Del_\Phi:\norm{h_\Phi}\le T}$ is contained 
in some compact set $K=K_T$ of shapes. Lemma~\ref{dist lemma} below is a converse statement of some sort for tight simplex sets 
with associated lattice in a given compact set of shapes.

\begin{lemma}\label{lemmadistortion}
For any $M>0$ and any compact set of shapes $K$, if $x\in\Om_{M,K}$ and $\Phi\subset
\Del_x$ is an $M$-tight simplex set then the distortion of $\Phi$ satisfies $\norm{h_\Phi}\ll_{M,K} 1$.
\end{lemma}
We note that in applications, most often one knows that $x\in\Om_M$ belongs to $\Om_{M,K}$ because one exhibits
a simplex set $\Phi$ of bounded distortion, which makes the conclusion of the lemma automatic. Nonetheless we
chose to give it here as it seems reasonable to expect an application in which one knows that $\Del_\Phi$ belongs to a compact set of shapes without knowing that $\Phi$ is responsible for that.
\begin{proof}
Let $\Phi$ be as in the statement. The assumption $\Del_\Phi\in K$, implies that there exists a 
basis $\set{\mb{t}_\ell}_1^n$ of $\Del_\Phi$  such 
that the linear map $h$ sending $\mb{b}_\ell$ to $\av{\Del_\Phi}^{-\frac{1}{n}}\mb{t}_\ell$, ($\ell=1\dots n$), satisfies 
$\norm{h}\ll_K1$ . Denote $\Phi'\defi \set{\mb{t}_\ell}_1^{n+1}$ where $\mb{t}_{n+1}=-\sum_1^n\mb{t}_\ell$ so that 
$h_{\Phi'}=h$. Let $g$ be a linear map sending $\Phi'$ to $\Phi$ and note that $h_\Phi=gh_{\Phi'}$. It follows that we 
would be done if we show that $\norm{g}\ll_{M,K}1$. To see this we note that on the one hand by Lemma~\ref{dist lemma}\eqref{dist1.6}, 
for any $\mb{t}\in \Phi'$, 
$\norm{\mb{t}}\gg_K \xi_{\Phi'}$ and therefore $\norm{g}\ll_K \frac{\xi_\Phi}{\xi_{\Phi'}}$. On the other hand, by Proposition~\ref{length bound prop} and our assumption we have that $e^{-\frac{n}{2}\xi_{\Phi'}}\ll \ell(x)\le 
M e^{-\frac{n}{2}\xi_\Phi}$ and therefore $\xi_\Phi-\xi_{\Phi'}\ll_M1$. We conclude that $\norm{g}\ll_{M,K} 1+\frac{1}{\xi_{\Phi'}}$. 
Since $\xi_{\Phi'}\gg 1$ by Lemma~\ref{xi bound}, we get $\norm{g}\ll_{M,K}1$ as desired.
\end{proof}
The following lemma gives us the tool with which we bound the index of one lattice in another.
\begin{lemma}\label{index lemma}
Let $\Del<\Sig$ be lattices and suppose that there is a ball $B$ and vectors $\mb{v}_i$, $i=1,\dots,\ell,$ such that $\on{(i)}$
$\Sig \subset \set{\mb{v}_i}_{i=1}^\ell +B +\Del$, and $\on{(ii)}$ $\av{\Sig\cap (\mb{v}_i+B+\mb{u})}\le 1$ for any $\mb{u}\in\Del$, $1\le i\le \ell$. Then
$\br{\Sig:\Del}\le \ell$. 
\end{lemma}
\begin{proof}
Using assumption (i) we may define a map $\al:\Sig\to\Del$ by the following procedure: Given $\mb{t}\in\Sig$ we choose $1\le i\le \ell$ such that $\mb{t}\in \mb{v}_i+B+\mb{u}$ (with $\mb{u}\in\Del$), and set $\al(\mb{t})=\mb{u}$. Assumption (ii) implies that $\al$ is at most $\ell$ to 1. 
Note also that there is a constant $T_0$ such that $\norm{\al(\mb{t})}\le \norm{\mb{t}}+T_0$. We conclude that 
if we denote by $C_T$ the cube of side length $T$ centered at $0$, then for any $T>0$, 
$\av{\Sig\cap C_T}\le \ell\av{\al(\Sig)\cap C_{T+T_0}}$. 

Recall that for any lattice $\Del'$,
$\av{\Del'}^{-1}=\lim_T \frac{1}{T^n}\av{\Del'\cap C_T}$. It follows that 
$$\br{\Sig:\Del}=\frac{\av{\Del}}{\av{\Sig}}=\lim_T\frac{\av{\Sig\cap C_T}}{\av{\Del\cap C_T}}
\le \limsup_T \ell \frac{\av{\al(\Sig)\cap C_{T+T_0}}}{\av{\Del\cap C_T}}\le \ell.
$$
\end{proof}
We use the following definition to discuss what we referred to above as `topological complexity'.
\begin{definition}
Given $\del_1>0$ and a lattice $x$ we
say that $\mb{t}_1,\mb{t}_2\in\bR^d_0$ are
two \textit{distinct visits times} to $X^{\ge\del_1}$ if $a(\mb{t}_i)x$ belong to two distinct connected components of $Ax\cap X^{\ge\del_1}$.
\end{definition}
The following lemma will help us distinguish between visit times.
\begin{lemma}\label{distinct visits lemma}
Let $x\in \Om$, $B_1,B_2$ two balls in $\bR^d_0$ and $\del_1>0$. Assume that 
$\on{(i)}$ $\exists \mb{t}_i\in B_i$ with $a(\mb{t}_i)x\in X^{\ge\del_1}$, and $\on{(ii)}$
$\set{a(\mb{t})x :\mb{t}\in\partial B_1\cup\partial B_2}\subset X^{<\del_1}$. Then, if $\mb{t}_i$ are not distinct visit times to $X^{\ge\del_1}$ then there exist $\mb{u}_i\in B_i$ such that $\mb{u}_1-\mb{u}_2\in\Del_x$.
\end{lemma}
\begin{proof}
If $a(\mb{t}_i)x$ are in the same component of $Ax\cap X^{\ge \del_1}$ then we may join them by a path within that component. As
the map $\mb{t}\mapsto a(\mb{t})x$ is a covering map, this path can be lifted to a one starting at $\mb{t}_1$. Because of assumption (ii), this lift is contained in $B_1$ and in particular, its end point $\mb{t}_2'$, which satisfies $a(\mb{t}_2')x=a(\mb{t}_2)x$,
is in $B_1$. We conclude that the difference $\mb{t}_2'-\mb{t}_2$ belongs to $\Del_x$ and so we may take $\mb{u}_1=\mb{t}_2'$
and $\mb{u}_2=\mb{t}_2$.  
\end{proof}

\begin{theorem}\label{index theorem}
Fix $M>1$, $K$ a compact set of shapes, and $\del_1\in(0,\del_0)$. For all but finitely many $x\in \Om_{M,K}$, if
$\Phi\subset \Del_x$ is an $M$-tight simplex set and we let
$$\mb{W}_\Phi''\defi\set{\mb{w}_\tau\in\mb{W}_\Phi': \exists \mb{t}\in \mb{w}_\tau+B_r, \;a(\mb{t})x\in X^{\ge\del_1}},$$
then,
\begin{enumerate}
\item\label{eq1524}  $\br{\Del_x:\Del_\Phi}\le \av{\mb{W}_\Phi''}$ (and of course, $\av{\mb{W}_\Phi''}\le n!$). 
\item\label{eq1525} If $d$ is prime then $\br{\Del_x:\Del_\Phi}=1$.
\item\label{eq1526} If $\mb{W}_\Phi''=\mb{W}_\Phi'$ then 
$\br{\Del_x:\Del_\Phi}=1$.
\item\label{eq1046} Under assumption~\eqref{eq1525} or \eqref{eq1526}, there are at least $\av{\mb{W}_\Phi''}$ distinct visits to $X^{\ge\del_1}$.
\end{enumerate}
\end{theorem}
\begin{remark}
We note that the assumption in~\eqref{eq1526} yields the best results, namely, we both get that $\Phi$ generates $\Del_x$ and that 
the orbit $Ax$ is as complex it could be. We point out two ways to verify this assumption in practice:
\begin{enumerate}
\item If the lattice $x$ belongs to $\Om_M(\eps,C)$ and the times $\mb{s}_\tau$ from Definition~\ref{omega epsilon c} are assumed 
to be bounded then it follows that for any $\tau\in\mb{W}_\Phi$, $a(\mb{w}_\tau)x\in X^{\ge\del_1}$ for some fixed $\del_1$. This 
is what happens in the construction discussed in~\S\ref{construction section} below.
\item Sometimes one can exploit symmetries of $x$ to bootstrap the existence of one $\mb{w}_\tau$ such that $a(\mb{w}_\tau)x\in X^{\ge\del_1}$ and upgrade it into the same statement for all $\mb{w}_\tau$. Such symmetries might exist if one assumes $x$ 
arises from a Galois extension.
\end{enumerate}
\end{remark}
\begin{proof}
Let $x\in \Om_{M,K}$ and $\Phi$ an $M$-tight simplex set for $x$. 
We use the notation of Proposition~\ref{eomprop} fixing some $\ka\in(0,1)$. 

Consider a time $\mb{t}$ such that $a(\mb{t})x\in X^{\ge\del_0}$. By Proposition \ref{eomprop}\eqref{eq1651}, 
(and the definition of $\mb{W}_\Phi''$), for all but finitely many $x\in\Om_M$
we have that 
\eqlabel{eq1514}{
\mb{t}+\Del_x\subset \mb{W}''_\Phi+ B_{r}+\Del_\Phi,
} 
where $r\asymp \av{\Del_x}^{\frac{\ka}{n}}$ is given in~\eqref{functions}.
Statement~\eqref{eq1524} will follow from Lemma~\ref{index lemma}, once 
we prove that for any $\tau$, the coset $\mb{t}+\Del_x$ can contain at most one point in each ball $\mb{w}_\tau+B_r+\mb{v}$ (for $\mb{v}\in\Del_\Phi$).

Assume such a ball contains two points of a coset of $\Del_x$. Then $\Del_x$ contains a non-trivial vector of size $\le 2r$. The assumption
that the shape of $\Del_\Phi$ is in $K$ implies, by Lemma~\ref{lemmadistortion}, that the distortion map $h_\Phi$ satisfies $\norm{h_\Phi}\ll_{M,K}1$, which in turn, by Lemma~\ref{dist lemma}\eqref{dist1.5}, implies that the distances between the balls composing $\mb{W}_\Phi+B_r+\Del_\Phi$
is $\gg_{M,K} \av{\Del_\Phi}^{\frac{1}{n}}-2r$. Since the latter expression is $>2r$ as soon as $\av{\Del_x}$ is big enough, we arrive at
a contradiction to the containment~\eqref{eq1514}.

We now establish~\eqref{eq1525} and \eqref{eq1526}. Let $\mb{t}$ be such that $a(\mb{t})x\in X^{\ge\del_0}$ as above. By \eqref{eq1514}, we may assume that 
\eqlabel{eq1550}{
\mb{t}\in \mb{w}_\tau+B_r,
}
for some $\mb{w}_\tau\in\mb{W}_\Phi''$. We will show that 
$\mb{t}+\Del_x\subset \mb{w}_\tau+B_r+\Del_\Phi$ which together with the already established fact by which each ball in 
the latter set contains at most one point of $\mb{t}+\Del_x$, will imply, by Lemma~\ref{index lemma}, that $\Del_x=\Del_\Phi$ as claimed. 

Assuming otherwise, there exists $\mb{w}_\tau\ne \mb{w}_\sig\in\mb{W}_\Phi''$ such that $(\mb{t}+\Del_x)\cap\pa{\mb{w}_\sig+B_r}\ne\varnothing$, which together
with~\eqref{eq1550} implies the existence of $\mb{s}\in B_{2r}$ such that 
\eqlabel{eq1157}{
\mb{v}\defi\mb{w}_\tau-\mb{w}_\sig+\mb{s}\in\Del_x.
}
To establish~\eqref{eq1525}, assume $d$ is prime and conclude from Lemma~\ref{reduction lemma}\eqref{claim5} that there
exists an integer $1\le k\le n$ such that $\mb{w}_\tau+k(\mb{w}_\tau-\mb{w}_\sig)\notin \mb{W}_\Phi+\Del_\Phi.$ As these objects are obtained from the corresponding ones for $\Phisym$ by applying a distortion map $h_\Phi$ followed by a dilation by a factor of 
$\av{\Del_\Phi}^{\frac{1}{n}}$, we conclude by Lemma~\ref{lemmadistortion}, Lemma~\ref{dist lemma}\eqref{dist1.5}, that 
$$\on{d}( \mb{w}_\tau+k(\mb{w}_\tau-\mb{w}_\sig), \mb{W}_\Phi+\Del_\Phi)\gg_{M,K} \av{\Del_x}^{\frac{1}{n}}.$$ 
Since
$\mb{w}_\tau+k(\mb{w}_\tau-\mb{w}_\sig) = \mb{t}+k\mb{v} +\mb{s}'$ for $\mb{s}'$ with $\norm{\mb{s}'}\ll r$ we have that 
$\on{d}(\mb{t}+k\mb{v},\mb{W}_\Phi+\Del_\Phi)\gg_{M,K} \av{\Del_x}^{\frac{1}{n}}$ which shows in particular, that for all but finitely many $x\in\Om_{M,K}$, $\mb{t}+k\mb{v}\notin\mb{W}_\Phi+B_r+\Del_\Phi$ and in turn, by Proposition~\ref{eomprop}\eqref{eq1651} gives that
$a(\mb{t}+k\mb{v})x\in X^{<\del}$. We arrive at the desired contradiction because 
$a(\mb{t}+k\mb{v})x=ua(\mb{t})x\in X^{\ge\del_0}$ (for some $u\in U$).

To establish~\eqref{eq1526} we proceed similarly. By Lemma~\ref{reduction lemma}\eqref{claim5} there exists $\tau'\in\crly{S}$ and 
$1\le k\le n$ such that $\mb{w}_{\tau'}+k(\mb{w}_\tau-\mb{w}_\sig)\notin \mb{W}_\Phi+\Del_\Phi$. 
As above, this implies that 
$\on{d}( \mb{w}_{\tau'}+k(\mb{w}_\tau-\mb{w}_\sig), \mb{W}_\Phi+\Del_\Phi)\gg_{M,K} \av{\Del_x}^{\frac{1}{n}}.$
By assumption there exists $\mb{t}'\in \mb{w}_{\tau'}+B_{r}$ such that $a(\mb{t}')x\in X^{\ge \del_1}$. As above, we conclude
that $\on{d}(\mb{t}'+k\mb{v},\mb{W}_\Phi+\Del_\Phi)\gg_{M,K} \av{\Del_x}^{\frac{1}{n}}$ which implies that 
$\mb{t}'+k\mb{v}\notin \mb{W}_\Phi+B_{r}+\Del_\Phi$ (except for finitely many possible exceptions). By Proposition~\ref{eomprop}\eqref{eq1651} we conclude that $a(\mb{t}'+k\mb{v})x\in X^{<\del}$ which contradicts the fact that for some $u\in U$, $a(\mb{t}'+k\mb{v})x=ua(\mb{t}')x\in X^{\ge\del_1}$ (as soon as $\av{\Del_x}$ is big enough).

Statement~\eqref{eq1046} follows from Lemma~\ref{distinct visits lemma} and the above considerations: 
For any $\mb{w}_\tau\in\mb{W}_\Phi''$ let $\mb{s}_\tau\in B_r$ be such that $a(\mb{w}_\tau+\mb{s}_\tau)x\in X^{\ge\del_1}$. 
By Lemma~\ref{distinct visits lemma}, if amongst these times are two non-distinct visit times, then $\Del_x$ must contain
a vector as in~\eqref{eq1157} which results in a contradiction as shown above.
\end{proof}
\subsection{} We summarize the results above for convenience of reference:
\begin{corollary}\label{maincor}
Fix $M,C\ge 1$, $\eps>0$ and $K$ a compact set of shapes. Then the following hold:
\begin{enumerate}
\item Any sequence of (distinct) orbits $Ax_k$ of lattices $x_k\in\Om_M$  satisfies the conclusion of Corollary~\ref{t.0901}.
\item Any sequence of (distinct) orbits $Ax_k$ of lattices $x_k\in\Om_M(\eps,C)$ satisfies the conclusion of Theorem~\ref{t.0932}.
\item For any sequence of (distinct) orbits $Ax_k$ of lattices $x_k\in\Om_M(\eps,C)\cap \Om_{M,K}$, if $\Phi_k$ is an $M$-tight
simplex set for $x_k$ satisfying Definition~\ref{omega epsilon c} such that the times $\mb{s}_\tau$ appearing there are uniformly bounded, then,  for any $\del_1\in (0,\del_0)$, and for all large $k$, $\on{(i)}$ $\Phi_k$ generates $\Del_x$ and, $\on{(ii)}$ 
$Ax_k\cap X^{\ge \del_1}$ has at least $n!$ distinct connected components.
\end{enumerate}
\end{corollary}

\section{The construction}\label{construction section}
In this section we construct infinite families of lattices satisfying the requirements of Corollary~\ref{maincor} and doing so, we
prove Theorem~\ref{t.0932}.
\subsection{The polynomials}
Fix $\eta>0$ and define 
\eqlabel{Zr}{
\bZ^d(\eta)\defi\set{\mb{m}\in\bZ^d: \forall i,j, \;\frac{\av{m_i-m_j}}{\norm{\mb{m}}}\ge \eta}.
}
Thus, $\bZ^d(\eta)$ consists of those integer vectors which are projectively bounded away from the ${d \choose 2}$ hyperplanes of equal coordinates. For $\mb{m}\in\bZ^d(\eta)$ define 
$$p_{\mb{m}}(x)\defi\prod_1^d(x-m_j)-1.$$ 
\begin{lemma}\label{basic lemma}
Fix $\eta>0$. For all but finitely many $\mb{m}\in\bZ^d(\eta)$, the polynomial $p_{\mb{m}}$ is irreducible over $\bQ$ and has $d$ 
real roots $\theta_j=\theta_j(\mb{m})$, $j=1\dots d$
which when put in the correct order satisfy $\theta_{j}= m_j+O_\eta(\norm{\mb{m}}^{-n})$. 
\end{lemma}
\begin{proof}
Let $\mb{m}\in\bZ^d(\eta)$ and let $\theta\in\bC$ be a root of $p_{\mb{m}}$; that is, 
\begin{equation}\label{productis1}
\prod_{j=1}^{d}(\theta-m_j)=1.
\end{equation} 
As $\av{m_i-m_j}\ge \eta\norm{m}$, for all but finitely many $\mb{m}$, 
$\theta$ can satisfy $\av{\theta-m_j}<1$ for at most one choice of $j$. In turn, because of the above 
equation, this inequality must hold for exactly one such $j\defi j_\theta$. 
It follows that for $j\ne j_\theta$ we have $\av{\theta- m_j}= \av{m_{j_\theta}-m_j}+O(1)=O_\eta(\norm{\mb{m}})$.
Going back to the equation $\av{\theta-m_{j_\theta}}=\prod_{j\ne j_\theta} \av{\theta-m_j}^{-1}$ 
we conclude that $\av{\theta-m_{j_\theta}}=O_\eta(\norm{\mb{m}}^{-n})$.

%Write $p_k(x)=\prod_{j=1}^n(x-\theta_{k,j})$ where $\theta_{k,j}$ are the complex roots of $p_k$. 
%
We claim that the map $\theta\mapsto j_\theta$ is one to one and onto from the set of complex zeros of $p_\mb{m}$
to $\set{1,\dots,d}$. Assume this is not the case and write $p_{\mb{m}}(x)=\prod_{j=1}^{d}(x-\theta_{j})$. Then
there are $j_1\ne j_2$ such that 
$\ell\defi j_{\theta_{j_1}}=j_{\theta_{j_2}}$. We conclude that the value
\begin{equation}\label{e1010}
\av{p_{\mb{m}}(m_\ell)}=\prod_{j=1}^{d}\av{m_\ell-\theta_{j}}
\end{equation}
is a non-zero integer on the one hand, but on the other hand, the two factors in the product~\eqref{e1010} that correspond to 
$j_1,j_2$ are $O_\eta(\norm{\mb{m}}^{-n})$ while all the other $n-1$ factors are $O_\eta(\norm{\mb{m}})$. We conclude that
for all but finitely many $\mb{m}$, the value in~\eqref{e1010} is $<1$ which gives a contradiction. 

Thus, we may order the roots $\theta_{j}$ of $p_{\mb{m}}$ in such a way so that 
\begin{align}
\label{e14351} &\av{\theta_{j}- m_j}=O_\eta(\norm{\mb{m}}^{-n}),\\
\label{e14352}\textrm{for $\ell\ne j$, }&\av{\theta_{j}-m_\ell}=O_\eta(\norm{\mb{m}}).
\end{align}
This immediately implies that (as soon as $\norm{\mb{m}}$ is large enough), all the roots are real because $p_\mb{m}$ is real so complex roots come in conjugate pairs. 

A similar argument gives the irreducibility of $p_\mb{m}$: If  $p_\mb{m}$ is reducible, then there must exist a proper subset $I$ of $\set{1\dots d}$ such that 
$q(x)= \prod_{j\in I}(x-\theta_{j})$ is a polynomial over the integers. 
We choose $\ell\in I$ and consider the value 
$
\av{q(m_\ell)}=\prod_{j\in I}\av{m_\ell-\theta_{j}}.
$ 
This number should be a (non-zero) integer by assumption but
on the other hand, 
the term in the product which corresponds to $j=\ell$ is $O_\eta(\norm{\mb{m}}^{-n})$ while all the other terms
are $O_\eta(\norm{\mb{m}})$. We derive a contradiction (as soon as $\norm{\mb{m}}$ is large enough), because there are at most $n-1$ such terms
(as the cardinality of $I$ is assumed to be $\le n$).
\end{proof}
\subsection{The lattices}
Fix $\eta>0$ and $\mb{m}\in\bZ^d(\eta)$ with $\norm{\mb{m}}$ large enough so that Lemma~\ref{basic lemma} applies. 
Fix an abstract root $\theta=\theta(\mb{m})$ of $p_\mb{m}$ and let $F_\mb{m}\defi\bQ(\theta)$ be the number field generated by it. By Lemma~\ref{basic lemma}, $F_\mb{m}$ is a totally real number field of degree $d$ over $\bQ$. By Lemma~\ref{basic lemma}, we can order
the embeddings $\sig_1\dots \sig_d$ of $F_\mb{m}$ into $\bR$  in such a 
way so that $\theta_{j}\defi \sig_j(\theta)$ satisfies $\theta_{j}=m_j+O_\eta(\norm{\mb{m}}^{-n})$. Let 
$$\Lam_\mb{m}=\spa_\bZ\set{1,\theta,\dots,\theta^n}\subset F_\mb{m}.$$
Then, $\Lam_\mb{m}$ is a full module in $F_\mb{m}$; that is, it is a free $\bZ$-module of rank $d$ that spans $F_\mb{m}$ over $\bQ$ (in fact it is 
the order $\bZ[\theta]$). Let $\bm{\sig}:F_\mb{m}\hookrightarrow\bR^d$ denote the map
whose $j$'th coordinate is $\sig_j$.
Then it is a well known fact that the image under $\bm{\sig}$ of a full module in $F_\mb{m}$ is a lattice in $\bR^d$. In particular, if $D_\mb{m}\defi \on{covol}(\bm{\sig}(\Lam_\mb{m}))$ then 
$$x_\mb{m}\defi D_\mb{m}^{-\frac{1}{d}}\bm{\sig}(\Lam_\mb{m})$$ 
is a unimodular lattice in $\bR^d$; i.e.\  $x_\mb{m}\in X$. 
\subsection{} We now show that Corollary~\ref{maincor} applies for the lattices constructed above and by that complete 
the proof of the statements in \S\ref{introduction section}.
\begin{proposition}\label{prop1238}
For any $\eta>0$ there exists $M,C>1$, $\eps>0$ and a compact set of shapes $K$ such that for all but finitely many
$\mb{m}\in\bZ^d(\eta)$, the lattice $x_\mb{m}$ belongs to $\Om_M(\eps,C)\cap \Om_{M,K}$. Moreover, we exhibit an $M$-tight simplex set $\Phi_\mb{m}$
for $x_\mb{m}$ satisfying Definition~\ref{omega epsilon c} such that the times $\mb{s}_\tau$ appearing there are uniformly bounded.
\end{proposition}
\begin{proof}
Let $\mb{m}\in\bZ^d(\eta)$ be of large enough norm so that Lemma~\ref{basic lemma} applies.
Our goal is to find a simplex set $\Phi$ for $x_\mb{m}$ that will be $M$-tight and will satisfy the requirements of $\Om_M(\eps,C)$ and $\Om_{M,K}$. We start by estimating $\ell(x_\mb{m})$. 

As $\mb{1}\in\bm{\sig}(\Lam_\mb{m})$ we have that $v_\mb{m}\defi D_\mb{m}^{-\frac{1}{d}}\mb{1}\in x_\mb{m}$. 
Let us calculate $D_\mb{m}$. By definition $\bm{\sig}(\Lam_\mb{m})$ is the lattice spanned by the columns of the (Vandermonde)
matrix $\pa{\sig_i(\theta^{j-1})}=\pa{\theta_i^{j-1}}$; 
\begin{align}
\nonumber D_\mb{m} = \det\smallmat{1&\theta_{1}&\theta_{1}^2&\dots& \theta_{1}^{n}\\
1&\theta_{2}&\theta_{2}^2&\dots&\theta_{2}^{n}\\
& &\ddots& & \\
1&\theta_d&\theta_d^2&\dots&\theta_d^{n}
}&=\prod_{i<j}(\theta_{j}-\theta_{i})
=\prod_{i<j}(m_j-m_i +O_\eta(\norm{\mb{m}}^{-n}))\\
\nonumber &=\norm{\mb{m}}^{{d\choose 2}}\prod_{i<j}(\frac{m_j-m_i}{\norm{\mb{m}}}+O_\eta(\norm{\mb{m}}^{-d}))\\
\label{eqdiscest}&=\norm{\mb{m}}^{d\choose 2}O(1).
\end{align}
We conclude that 
\eqlabel{eq:lovk}{
\ell(x_\mb{m})\le \norm{v_\mb{m}}\ll \norm{\mb{m}}^{-\frac{n}{2}}.
}

We now turn to investigate $\Del_{x_\mb{m}}$ in order to find a suitable simplex set $\Phi$ in it. 
The link with the action of the diagonal group on $X$ originates from the following observation.
Let $\al\in F_\mb{m}$ and let $R_\al:F_\mb{m}\to F_\mb{m}$ denote multiplication by $\al$.
Then the following diagram commutes:
\begin{equation}\label{diagram1}
\xymatrix{
F_\mb{m} \ar[d]_{R_\al}\ar[r]^{\bm{\sig}}&\bR^d\ar[d]^{\diag{\bm{\sig}(\al)}}\\
F_\mb{m}\ar[r]_{\bm{\sig}}& \bR^d
}
\end{equation}
It is clear that for any $1\le \ell\le d$, ${\om}_\ell\Lam_\mb{m}\subset\Lam_\mb{m}$, where 
\eqlabel{omegahat}{
\om_\ell=\om_\ell(\mb{m})\defi \theta-m_\ell.
}
In fact, we have that $\om_\ell\Lam_\mb{m}=\Lam_\mb{m}$ because $\om_\ell$ is a unit in the ring of integers. To see this,
note that~\eqref{productis1} shows that both $\om_\ell$ and its inverse are algebraic integers.
It follows from~\eqref{diagram1} that $\diag{\bm{\sig}(\om_\ell)}\in UA$ stabilizes $x_\mb{m}$ and therefore, by~\eqref{deltax},
\eqlabel{eqtl}{
\mb{t}_\ell\defi\log(\av{\diag{\bm{\sig}(\om_\ell)}})\in\Del_{x_\mb{m}}.
}
We claim that $\Phi_\mb{m}=\Phi\defi\set{\mb{t}_\ell}_{\ell=1}^d$ forms the desired simplex set. 
The fact that $\sum_\ell \mb{t}_\ell=0$ follows from~\eqref{productis1}. From~\eqref{e14351}, \eqref{e14352} we 
conclude several things: First we note 
that the set $\frac{1}{\log(\norm{\mb{m}})}\Phi_\mb{m}$ converges as $\norm{\mb{m}}\to\infty$ to $\Phisym$. This implies in particular, that
$\Phi$ spans $\bR^d_0$ and in turn implies that it is a simplex set. Moreover, it shows that the distortion map $h_\Phi$ converges 
to the identity and so there exists a compact set $K$ of shapes such that $\Del_\Phi\in K$.  Second, we conclude that
\eqlabel{eq1056}{
\xi_\Phi = \max_{1\le \ell\le d}\max \mb{t}_\ell\le \log \norm{\mb{m}}+ O_\eta(1).
} 
Combining~\eqref{eq:lovk} and \eqref{eq1056} we conclude one can choose $M=M(\eta)$ such that 
$\ell(x_\mb{m})\le Me^{-\frac{n}{2}\xi_\Phi}$. Summarizing, we have established that for all but finitely many $\mb{m}\in\bZ^d(\eta)$,
$x_\mb{m}\in \Om_{M,K}$ and that $\Phi_\mb{m}$ is an $M$-tight simplex set for $x_\mb{m}$ with $\Del_\Phi$ in $K$.

We now turn to study the lattices $a(\mb{w}_\tau)x_\mb{m}$ for $\tau\in\crly{S}$ in order to verify Definition~\ref{omega epsilon c}.%
%
%
%
%
%A
%
%
%
%
%

It would be convenient to introduce the following notation: 
Given two matrices $A=(a_{ij}),B=(b_{ij})$ (say, depending on the parameter $\eta$), 
we write $$A\ll^{\on{av}}_\eta B$$ to indicate that there exists a constant $c=c(\eta)>0$, such that $\av{a_{ij}}\le c\av{b_{ij}}$.

 For simplicity let us assume that $\tau$ is the identity permutation and denote $\mb{w}\defi \mb{w}_{\on{id}}^{\Phi_\mb{m}}.$ 
We begin by estimating $a(\mb{w})$ and then continue to find a suitable matrix representing the lattice
$a(\mb{w})x_\mb{m}$: 
By definition~\eqref{def:w},  
\eq{
\mb{w}=\sum_{\ell= 1}^{d}\frac{\ell-1}{d}\cdot\mb{t}_\ell
}
and so by~\eqref{eqtl},
\begin{align}
\nonumber a(\mb{w})&=\textstyle\prod_{\ell=1}^{d}\pa{ \diag{\av{\sig_1(\om_\ell)},\dots,\av{\sig_{d}(\om_\ell)}}}^{\frac{\ell-1}{d}}\\
\nonumber &=\textstyle\diag{\av{\sig_1( \prod_{\ell=1}^{d}\om_\ell^{\ell-1})},\dots,\av{\sig_{d}(\prod_{\ell=1}^{d}\om_\ell^{\ell-1})}}^{\frac{1}{d}}\\
\label{eq2336}&\ll^{\on{av}}_\eta \norm{\mb{m}}^{-\frac{n}{2}}\diag{\norm{\mb{m}}^n, \norm{\mb{m}}^{n-1},\dots, \norm{\mb{m}}, 1},
\end{align}
%}
where the last estimate follows from~\eqref{omegahat}, \eqref{e14351}, \eqref{e14352}.

Recall that the lattice $x_\mb{m}$ has a basis corresponding to the basis of $\Lam_\mb{m}$ given by $\set{\theta^j}$, $j=0\dots n$.
We make the following change of base of $\Lam_\mb{m}$ 
\eq{
\set{1,\theta,\dots,\theta^n}\to\set{1,\om_1,\om_1\om_2,\dots,\om_1\cdots\om_n}.
}
 The corresponding basis of $x_\mb{m}$ is given by the columns of the
matrix 
\eq{
f_\mb{m}\defi
 D_\mb{m}^{-\frac{1}{d}}\mat{
 1&\sig_1(\om_1)&\dots&\sig_1(\om_1\cdots \om_n)\\
 1&\sig_2(\om_1)&\dots&\sig_2(\om_1\cdots \om_n)\\
 \vdots&             &         &       \vdots                        \\ 
 1&\sig_{d}(\om_1)&\dots&\sig_{d}(\om_1\cdots \om_n)
 }
}
%\eq{
%h_k\defi D_k^{-\frac{1}{n+1}}\mat{\sig_i(\prod_0^{j} \om_\ell)}.
%}
Recall that by~\eqref{eqdiscest},  $D_\mb{m}^{-\frac{1}{d}}= O(\norm{\mb{m}}^{-\frac{n}{2}})$ and by~\eqref{e14351}, \eqref{e14352} that 
\eq{
\sig_i(\om_\ell)=\Bigg\{
\begin{array}{ll}
O_\eta(\norm{\mb{m}})  & \textrm{ if }\ell\ne i;\\
O_\eta(\norm{\mb{m}}^{-n}) & \textrm{ if }\ell=i
\end{array}
\Longrightarrow\;
\sig_i(\prod_1^{j-1} \om_\ell)=\Bigg\{
\begin{array}{ll}
O_\eta(\norm{\mb{m}}^{j-n-1}) & \textrm{ if } i< j;\\
O_\eta(\norm{\mb{m}}^{j-1}) & \textrm{ if } i\ge j.
\end{array}
}
In other words,
\eqlabel{eq2335}{f_\mb{m}\ll^{\on{av}}_\eta
\norm{\mb{m}}^{-\frac{n}{2}}\smallmat{
1 & \norm{\mb{m}}^{-n} & \norm{\mb{m}}^{-(n-1)} & \norm{\mb{m}}^{-(n-2)} &\dots & \norm{\mb{m}}^{-1}\\
1 & \norm{\mb{m}}         & \norm{\mb{m}}^{-(n-1)} & \norm{\mb{m}}^{-(n-2)} &\dots & \norm{\mb{m}}^{-1}\\
1 & \norm{\mb{m}}         & \norm{\mb{m}}^2          & \norm{\mb{m}}^{-(n-2)} &\dots & \norm{\mb{m}}^{-1}\\
1 & \norm{\mb{m}}         & \norm{\mb{m}}^2          & \norm{\mb{m}}^3          &\dots & \norm{\mb{m}}^{-1}\\
\vdots &  \;  & \vdots      & \vdots       &\ddots& \vdots \\
1 & \norm{\mb{m}}         & \norm{\mb{m}}^2          &||\mb{m}||^3          &\dots  &\norm{\mb{m}}^n
}.
}
So, by~\eqref{eq2336}, \eqref{eq2335} we see that the matrix  $a(\mb{w})f_\mb{m}$ which represents $a(\mb{w})x_\mb{m}$ satisfies
\eqlabel{eq1413}{
a(\mb{w})f_\mb{m}\ll^{\on{av}}_\eta
\smallmat{
1 & \norm{\mb{m}}^{-n} & \norm{\mb{m}}^{-(n-1)} & \norm{\mb{m}}^{-(n-2)} &\dots & \norm{\mb{m}}^{-1}\\
\norm{\mb{m}}^{-1} & 1         & \norm{\mb{m}}^{-n} & \norm{\mb{m}}^{-(n-1)} &\dots & \norm{\mb{m}}^{-2}\\
\norm{\mb{m}}^{-2} & \norm{\mb{m}}^{-1}         & 1          & \norm{\mb{m}}^{-n} &\dots & \norm{\mb{m}}^{-3}\\
\norm{\mb{m}}^{-3} & \norm{\mb{m}}^{-2}         & \norm{\mb{m}}^{-1}          & 1          &\dots & \norm{\mb{m}}^{-4}\\
\vdots &  \;  & \vdots      & \vdots       &\ddots& \vdots \\
\norm{\mb{m}}^{-n} & \norm{\mb{m}}^{-(n-1)}         & \norm{\mb{m}}^{-(n-2)}          & \norm{\mb{m}}^{-(n-3)}          &\dots  & 1
}.
}
Write $a(\mb{w})f_\mb{m}=(c_{ij})$. Working with the inequalities in~\eqref{eq1413} and the fact that $\det a(\mb{w})f_\mb{m}=1$
it is straightforward to deduce that in the brut-force decomposition
$$(c_{ij}) = 
\smallmat{ 
\frac{c_{11}}{\av{c_{11}}}(\prod_1^d c_{\ell\ell})^{\frac{1}{d}}&\cdots& \frac{c_{1d}}{\av{c_{dd}}}(\prod_1^d c_{\ell\ell})^{\frac{1}{d}}\\
\vdots &  \frac{c_{ij}}{\av{c_{jj}}}(\prod_1^d c_{\ell\ell})^{\frac{1}{d}}&\vdots\\
 \frac{c_{d1}}{\av{c_{11}}}(\prod_1^d c_{\ell\ell})^{\frac{1}{d}}&\cdots&  \frac{c_{dd}}{\av{c_{dd}}}(\prod_1^d c_{\ell\ell})^{\frac{1}{d}}
 }
\smallmat{
\frac{\av{c_{11}}}{(\av{\prod_1^d c_{\ell\ell}})^{\frac{1}{d}}}&\cdots&0\\
\vdots&\ddots&\vdots\\
0&\cdots& \frac{\av{c_{dd}}}{(\av{\prod_1^d c_{\ell\ell}})^{\frac{1}{d}}}}$$
if we denote the matrices on the right by $g,a(\mb{s})$ respectively, then $\norm{g-I}\ll_\eta \norm{\mb{m}}^{-1}$, $\norm{\mb{s}}\ll_\eta 1$.

Combining \eqref{eq1056} and Lemma~\ref{dist lemma} we get 
that $\Del_{x_\mb{m}}\le \Del_\Phi\ll \xi_\Phi^n\ll(\log\norm{\mb{m}})^n$ and so choosing $\eps, C$ appropriately we can make sure that
 $\norm{g-I}\le Ce^{-\av{\Del_x}^\eps}$. This concludes the proof that $x_\mb{m}\in\Om_M(\eps,C)$ and hence the proof of the Proposition.

\end{proof}

\bibliographystyle{plain}
\bibliography{elonbib}
\end{document}